\documentclass[a4paper,10pt]{article}
\usepackage[utf8]{inputenc}
\usepackage[T1]{fontenc}
\usepackage{lmodern}
\usepackage{a4wide}
\usepackage{geometry}
\usepackage{graphicx}
\usepackage{amsmath,amsfonts,amsthm,amssymb,here,dsfont,hyperref}
\usepackage{color}
\usepackage{version}
\usepackage{todonotes}
\usepackage[lofdepth,lotdepth]{subfig}
\usepackage{wrapfig}
\usepackage{fullpage}
\usepackage{xcolor}
\usepackage{multirow}
\usepackage{caption}
\usepackage{natbib}
\usepackage{nicefrac}       
\usepackage{microtype}      
\usepackage{algorithm, algorithmic}
\usepackage{diagbox}
\usepackage{colortbl}

\newtheorem{theo}{Theorem}[section]
\newtheorem{definition}[theo]{Definition}
\newtheorem{prop}[theo]{Proposition}

\newtheorem{coro}[theo]{Corollary}
\newtheorem{lem}[theo]{Lemma}
\newtheorem{rem}[theo]{Remark}

\newtheorem{ass}[theo]{Assumption}

\newcommand{\one}{\mathds{1}}

\newcommand{\conv}{\mathrm{conv}}

\newcommand\smallO{
  \mathchoice
    {{\scriptstyle\mathcal{O}}}
    {{\scriptstyle\mathcal{O}}}
    {{\scriptscriptstyle\mathcal{O}}}
    {\scalebox{.7}{$\scriptscriptstyle\mathcal{O}$}}
  }

\title{Fairness guarantee in multi-class classification}
\begin{document}

\title{Fairness guarantee in multi-class classification}
 \author{Christophe Denis$^{(1)}$,
 Romuald Elie$^{(1)}$,
 Mohamed Hebiri$^{(1)}$, 
and Fran\c{c}ois Hu$^{(2)}$\\
\small{$^{(1)}$LAMA, UMR-CNRS 8050, Universit\'e Gustave Eiffel}\\
\small{$^{(2)}$ Département de Mathématiques et Statistique, Université de Montreal}}
 \date{}
 
\maketitle

\begin{abstract}
Algorithmic Fairness is an established area of machine learning, willing to reduce the influence of hidden bias in the data. Yet, despite its wide range of applications, very few works consider the multi-class classification setting from the fairness perspective.
We focus on this question and extend the definition of approximate fairness in the case of \emph{Demographic Parity} to multi-class classification. We specify the corresponding expressions of the optimal fair classifiers. 
This suggests a plug-in data-driven procedure, for which we establish theoretical guarantees.  
The enhanced estimator is proved to  mimic the behavior of the optimal rule both in terms of fairness and risk. Notably, fairness guarantees are distribution-free. 
The approach is evaluated on both synthetic and real datasets and reveals very effective in decision making with a preset level of unfairness. In addition, our method is competitive (if not better) with the state-of-the-art in binary and multi-class tasks. 
\end{abstract}

\section{Introduction}
Algorithmic fairness has become very popular during the last decade ~\cite{zemel2013learning,lum2016statistical,calders2009building,zafar2017fairness,agarwal2019fair,Agarwal_Beygelzimer_Dubik_Langford_Wallach18,donini2018empirical,Chzhen_Denis_Hebiri_Oneto_Pontil19,chiappa2020general,barocas-hardt-narayanan} as it addresses an important social concern: mitigating historical bias contained in the data. This is a crucial issue in many applications such as loan assessment or criminal sentencing among others. 
The main objective in algorithmic fairness consists in reducing the influence of a sensitive attribute on a prediction. Several notions of fairness have been considered in the literature for binary classification~\cite{zafar2019fairness,barocas-hardt-narayanan}. All of them impose some independence condition between the sensitive feature and the prediction. 
This independence can be desired on some or all values of the label space, see \emph{Equality of odds} or \emph{Equal opportunity}~\citep{hardt2016equality}. In this paper, we focus on the well established \emph{Demographic Parity} (DP)~\citep{calders2009building} that requires the independence between the sensitive feature and the prediction function, while not relying on labels. DP has a recognized interest in many applications, such as loan agreement without gender attributes or crime prediction without ethnicity discrimination \citep{Hajian_2011_discrimation,KamiranZC13RemoveIllegal,barocas2014datas,feldman_certifying_2015}. 
Previously mentioned references consider either the regression or the binary classification frameworks, although most (modern) applications fall within the scope of multi-class classification (\emph{e.g.} image recognition or text categorization).
As an example, one might cite hiring tools based on Machine Learning (ML) models to give candidates one- to five-star ratings and favors men for software developers and other technical positions \citep{dastin2018amazon}.

The present work fills two gaps in the literature: i) it extends algorithmic fairness to the multi-class setting; ii) it properly studies the approximate fairness (also called as $\varepsilon$-fairness) from the theoretical point of view. Indeed, approximate fairness is known to be very efficient from a practical perspective~\cite{barocas-hardt-narayanan,zafar2019fairness}. Nevertheless, main existing theoretical results only focus on exact fairness constraints, that is, they do not allow for deviating from a perfectly fair algorithm.

\subsection{Main contributions}

Overall, we emphasize that the present paper considers both the theoretical and the practical aspects of approximate fairness
under the popular demographic parity constraint. Up to our knowledge, this is the first contribution that combines both aspects in the multi-class setting.  

We establish a closed formula of the optimal predictor for both exact and approximate fairness constraints. Our proposed procedure is a post-processing algorithm which relies on solving a constrained minimization problem. Specifically, in a first step we estimate the conditional probabilities of the output label given the sensitive attribute and the feature vectors, while a second step of the algorithm is dedicated to enforce fairness by shifting the estimated conditional probabilities in an optimal manner. 
We derive fairness and risk guarantees for our estimation procedure with explicit finite sample bounds. 
We also highlight the numerical performance of our algorithm and show that it performs as good as state-of-the-art multi-class methods for fair (or approximate fair) prediction.  
One of the main striking features of our procedure is that it can be applied to any off-the-shelf estimator of the conditional probabilities and it succeeds to enforce fairness at any pre-specified level.

We want to underline that the extensions, with respect to the existing literature, to multi-class and to approximate fairness are two theoretical aspects of the contribution. Both considerations involve additional technical arguments. In particular, dealing with approximate fairness is a new  technical challenge. It is worth noticing that even in the binary classification setting, the control of the unfairness of the algorithms has not been analyzed theoretically.

Let us now summarize our main contributions:
\begin{itemize}
    \item We provide an optimal solution for the multi-class problem under exact or approximate DP constraints.
    In particular, we derive a closed formula for the optimal (approximate) fair classifier.
    \item Based on this formula, we build a data-driven procedure that mimics the performance of the optimal rule both in terms of risk and fairness. Notably, our fairness guarantees are \emph{distribution-free}
    and are established both in expectation and with high probability.
    \item We also established rates of convergence for the resulting classifier {\it w.r.t.} a suitable risk that combines both the error rate and the unfairness measure. A salient point of our theoretical findings is that our procedure achieves fast rates of convergence under a Margin type assumption. 
    \item The approach is illustrated on several real and synthetic datasets with various bias levels. It provides robust and effective decision making rules with a preset level of unfairness.
\end{itemize}

\subsection{Related works}

There are mainly three ways to build fair prediction: i) \emph{pre-processing} methods mitigate bias in the data before applying classical ML algorithms, see for instance~\cite{adebayo2016iterative,calmon2017optimized,zemel2013learning}; ii) \emph{in-processing} methods reduce bias during training, see for instance~\citep{Agarwal_Beygelzimer_Dubik_Langford_Wallach18,Donini_Oneto_Ben-David_Taylor_Pontil18,agarwal2019fair}; 
iii) \emph{post-processing} methods enforce fairness after fitting, see for instance~\cite{Hardt_Price_Srebro16,chiappa2020general,Chzhen_Denis_Hebiri_Oneto_Pontil20TV,gouic2020price}. The present work falls within the last category. In a related study, \citep{Chzhen_Denis_Hebiri_Oneto_Pontil19} exhibits fair binary classifiers under \emph{Equal Opportunity} constraints. In contrast, we focus on the multi-class setting, while imposing \emph{DP} constraints and we also treat the case of approximate fairness.

Up to our knowledge, only few works consider fairness in the multi-class setting. In~\cite{Ye_Xie_2020_fairness}, the authors enforce fairness by sub-sample selection and is in-processing.
In contrast, we keep the whole sample and enforce fairness in a post-processing manner. Besides, from a high-level perspective, the procedure described  in~\citep{Ye_Xie_2020_fairness} imposes fairness on each component of the score function. It is clear that such methodology can be generalized to any convex empirical risk minimization (ERM) problem such as SVM or quadratic risk.  But, since the decision rule in the multi-class setting relies on the maximizer over scores, we rather directly impose fairness on the maximizer itself.

The multi-class framework is also considered in~\citep{Adversarial19,Tavkeretal2020, alghamdi2022beyond}. 
However, the authors in~\citep{Adversarial19,Tavkeretal2020} do not provide an explicit formulation of the optimal fair rule and their theoretical fairness guarantee is not distribution free. In addition, they only consider numerical experiments for binary classification. Finally, the recent work in ~\citep{alghamdi2022beyond} consider projecting an unfair classifier into a set of fair classifiers. However, as illustrated 
in Section~\ref{subsec:RealData}, their method seems to fail in \textit{exact} fairness.
Our method provides valuable benefits on all these aspects.

\subsection{Outline of the paper}

In Section~\ref{sec:FMCApprox}, we define the Demographic Parity constraint and the notion of  exact/approximate fair classifier in the multi-class classification setup. An explicit expression of the optimal fair classifier is also provided in Section~\ref{sec:FMCApprox}. The corresponding data-driven procedure together with its statistical guarantees on risk and fairness are presented in Section~\ref{sec:datadriven}. The numerical performance of the procedure is illustrated on both synthetic and real datasets in Section~\ref{sec:Evaluation}. The paper concludes with a discussion and perspective Section~\ref{sec:conclusion}.
For ease of readability, proofs and technical arguments are postponed to the Appendix of the paper.

\section{Multi-class classification with demographic parity}
\label{sec:FMCApprox}

Let $(X,S,Y)$ be a random tuple with distribution $\mathbb{P}$, where $X \in \cal{X}\subset$ $\mathbb{R}^{d}$, $S \in \mathcal{S}:=\{-1,1\}$ and $Y \in [K]:=\{1, \ldots, K\}$ with $K$ a fixed number of classes. The distribution of the sensitive feature $S$ is denoted by $(\pi_s)_{s \in \mathcal{S}}$, and we assume that $\min_{s \in \mathcal{S}} \pi_s > 0$, meaning that we have access to both sensitive groups with non zero probability.
A classification rule $g$ is a function mapping $\mathcal{X} \times \{-1,1\}$ onto $[K]$, and its performance is evaluated through the misclassification risk
\begin{equation*}
\mathcal{R}(g) := \mathbb{P}\left(g(X,S) \neq Y\right) \enspace.
\end{equation*}
For $k\in[K]$, we denote $p_k(X,S):= \mathbb{P}\left(Y=k |X,S\right)$ the conditional probabilities. 
Recall that a Bayes classifier minimizing the misclassification risk $\mathcal{R}(\cdot)$ over the set $\mathcal{G}$ of all classifiers and is then given by
\begin{equation*}
g^*(x,s) \in \arg\max_{k} p_k(x,s) \;, \quad \mbox{for all } (x,s)\in \cal{X}\times \cal{S} \enspace.
\end{equation*}
We introduce in Section~\ref{subsec:demParity} the Demographic parity constraint as well as the definition of an approximate fair classifier. The characterization of the optimal fair
classifier and its main properties are provided in Section~\ref{subsec:optimalFairClassifier}.

\subsection{Demographic parity}
\label{subsec:demParity}

We consider DP constraint~\citep{calders2009building} that asks for independence of the prediction function from the sensitive feature $S$.
This definition naturally extends the DP constraint considered in binary classification \citep{agarwal2019fair,chiappa2020general,gordaliza2019obtaining,jiang2019wasserstein,oneto2019general}. 

Approximate fairness, also referred to as $\varepsilon$-fairness, is highly popular from a practical perspective, in particular when a strict fairness constraint strongly deflates the accuracy of the method. In this context, the user is allowed to adjust the fairness constraint if relevant or needed. Of course, such modularity has a cost: the solution is less fair than the exact fair one. Moreover, the chosen unfairness level has no convincing interpretation. Without clear justification, some empirical rules exist such as the forth-firth that tolerates an unfairness of $0.2$~\citep{HolzerHolzer00,Collins07,feldman_certifying_2015}. 
In this section, we consider \emph{approximate} fairness setting without discussing the issue of properly selecting of the unfairness level $\varepsilon$.

We define the notion of $\varepsilon$-fairness in the particular case of Demographic Parity.

\begin{definition}[$\varepsilon$-fairness \emph{w.r.t.} DP]
\label{def:espUnfairnessMeasure} 
The unfairness of a classifier $g\in\mathcal{G}$ is quantified by
\begin{equation*}
\begin{aligned}
\mathcal{U}(g) := & \max_{k \in [K]} \left|\mathbb{P}\left(g(X,S) = k | S=1 \right)  - \mathbb{P}\left(g(X,S) = k| S=-1\right) \right| \enspace .
\end{aligned}
\end{equation*} 
A classifier $g$ is $\varepsilon$-fair if and only if 
$\mathcal{U}(g) \leq  \varepsilon$. In particular, $\varepsilon=0$ means that $g$ is exactly fair.
 \end{definition}
Alternative measures of unfairness could be considered. The maximum can for instance be replaced by a summation over $k \in [K]$.
While both measures have their advantages, picking the maximum simplifies fairness evaluation in empirical studies.

\subsection{Optimal fair classifier}
\label{subsec:optimalFairClassifier}

Our goal is to derive an explicit formulation of the optimal $\varepsilon$-fair classifiers \emph{w.r.t.} the misclassification risk, denoted by $g^*_{\varepsilon-{\rm fair}} $, which solves $
\min_{g\in \mathcal{G}_{\varepsilon-{\rm fair}}} \mathcal{R}(g)$ where $\mathcal{G}_{\varepsilon-{\rm fair}}$ is the set of all $\varepsilon$-fair prediction functions. 
Its computation requires to properly balance misclassification risk together with fairness criterion. 
The first step is to write the Lagrangian of the above problem: for $\lambda^{(1)}=(\lambda_1^{(1)},\ldots,\lambda_K^{(1)})\in\mathbb{R}_{+}^K$ and $\lambda^{(2)}=(\lambda_1^{(2)},\ldots,\lambda_K^{(2)})\in\mathbb{R}_{+}^K$, we define the \emph{$\varepsilon$-fair-risk} as
\begin{small}
\begin{equation*}
\begin{aligned}
 \mathcal{R}_{\lambda^{(1)}, \lambda^{(2)}}(g) :=  \mathcal{R}(g) & 
 + \sum_{k = 1}^K \lambda_k^{(1)} 
[\mathbb{P}\left(g(X,S) = k| S=1\right)
 - \mathbb{P}\left(g(X,S) = k| S=-1\right) - \varepsilon] \\
 & + \sum_{k = 1}^K \lambda^{(2)}_k  [\mathbb{P}\left(g(X,S) = k| S=-1\right)
 - \mathbb{P}\left(g(X,S) = k| S=1\right) - \varepsilon] \enspace.
\end{aligned}
\end{equation*}
\end{small}
In order to characterize the optimal fair classifier, we also require the following technical condition.
\begin{ass}[Continuity assumption]
\label{ass:continuity}
The mapping
$t \mapsto \mathbb{P}\left(p_k(X,S)-p_j(X,S) \leq t |S=s\right)$ is assumed continuous, for any $ k,j  \in[K]$ and $s\in \mathcal{S}$.
\end{ass}
Assumption~\ref{ass:continuity} implies that the distribution of the differences $p_k(X,S)-p_j(X,S)$ has no atoms. It is required to derive a closed expression for $g^*_{\varepsilon-\rm fair}$ and insures an accurate calibration of the fairness at the prescribed level. Notice that in the binary case ($K=2$), it boils down to the continuity of $t \mapsto \mathbb{P}\left(p_k(X,S) \leq t |S=s\right)$ considered in~\citep{Chzhen_Denis_Hebiri_Oneto_Pontil19}. However when $K \geq 3$, these two conditions differ and we stress that Assumption~\ref{ass:continuity} is a well tailored condition for the multi-class problem.
\\
We are now in position to provide a characterization of optimal $\varepsilon$-fair classifier. 
\begin{theo}
\label{thm:equivalenceApprox}
Let $H: \mathbb{R}_{+}^{2K} \to \mathbb{R}$ be the function
\begin{small}
    \begin{equation*}
H(\lambda^{(1)}, \lambda^{(2)})  = \sum_{s \in \mathcal{S}} \mathbb{E}_{X|S=s}\left[\max_k\left(\pi_s p_k(X,s)-s(\lambda^{(1)}_k-\lambda^{(2)}_k)\right)\right] 
+ \varepsilon \sum_{k=1}^K (\lambda^{(1)}_k+\lambda^{(2)}_k)\enspace .
\end{equation*}
\end{small}
Let Assumption~\ref{ass:continuity} be satisfied and define $(\lambda^{*(1)} , \lambda^{*(2)}) \in \arg\min_{(\lambda^{(1)} , \lambda^{(2)}) \in \mathbb{R}_{+}^{2K}} H(\lambda^{(1)} , \lambda^{(2)}) $. 
Then, $g^*_{\varepsilon-{\rm fair}} \in \arg\min_{g\in \mathcal{G}_{\varepsilon-{\rm fair}}} \mathcal{R}(g)$ if and only if $g^*_{\varepsilon-{\rm fair}}  \in \arg\min_{g\in \mathcal{G}} \mathcal{R}_{\lambda^{*(1)} , \lambda^{*(2)}}(g)$.\\ In addition, for all $(x,s)\in {\cal{X}} \times {\cal{S}}$, we can rewrite the optimal classifier as
\begin{equation*}
g^*_{\varepsilon-{\rm fair}}(x,s) = \arg\max_{k \in [K]}  \left(\pi_s p_k(x,s)-s(\lambda^{*(1)}_k-\lambda^{*(2)}_k)\right).
\end{equation*}
\end{theo}
Theorem~\ref{thm:equivalenceApprox} entails a closed form expression of optimal fair classifiers, which is the bedrock of our procedure: any optimal fair classifier is simply maximizing scores, that are obtained by shifting the original conditional probabilities in a proper manner. The above result also points out that the optimum of the risk $\mathcal{R}$ over the class of fair classifiers also minimizes the fair-risk $\mathcal{R}_{\lambda^{*(1)}, \lambda^{*(2)}}$. Hence, by construction, $\mathcal{R}_{\lambda^{*(1)}, \lambda^{*(2)}}$ is a risk measure that efficiently balances both classification accuracy and unfairness.
An important consequence of the proof of Theorem~\ref{thm:equivalenceApprox} is the following proposition that more precisely characterizes the Lagrange multipliers ($\lambda^{*(1)}, \lambda^{*(2)}$), and the level of unfairness of the $\varepsilon$-fair predictor.

\begin{prop}
\label{prop:LagrangeCharac}
Let $\varepsilon \geq 0$. For each $k \in [K]$, we have that
$\lambda^{*(1)}_k\lambda^{*(2)}_k = 0 \;\; {\rm and} \;\; \lambda^{*(1)}_k +\lambda^{*(2)}_k \geq 0$.
Besides, if for some $k$
\begin{enumerate}
    \item[$i)$] $\lambda^{*(1)}_k > 0$, then $
\mathbb{P}_{X|S=1}\left(g^{*}_{\lambda^{*(1)},\lambda^{*(2)}}(X,S) = k \right) - \mathbb{P}_{X|S=-1}\left(g^*_{\lambda^{*(1)},\lambda^{*(2)}}(X,S) = k \right) = \varepsilon$,
    \item[$ii)$] $\lambda^{*(2)}_k > 0$, then 
$
\mathbb{P}_{X|S=1}\left(g^{*}_{\lambda^{*(1)},\lambda^{*(2)}}(X,S) = k \right) - \mathbb{P}_{X|S=-1}\left(g^*_{\lambda^{*(1)},\lambda^{*(2)}}(X,S) = k \right) = - \varepsilon$.
\end{enumerate}
\end{prop}
From the above result, we easily deduce the following corollary.
\begin{coro}
Let $\varepsilon \geq 0$. It holds that
\label{cor:UnfairnessCharac}
\begin{enumerate}
\item[$i)$] either the Bayes classifiers satisfies $\mathcal{U}(g^*) \leq \varepsilon$ and
then $g^* = g^*_{{\varepsilon-{\rm fair}}}$. In this case $\lambda^{*(1)} = \lambda^{*(2)}=0$;
\item[$ii)$] or the $\varepsilon$-fair classifier satisfies $\mathcal{U}(g^*_{{\varepsilon-{\rm fair}}}) = \varepsilon$.
\end{enumerate}
\end{coro}
A straightforward consequence of the above Proposition~\ref{prop:LagrangeCharac} and Corollary~\ref{cor:UnfairnessCharac} is that  
\begin{equation*}
0 \leq \mathcal{R}(g^*_{\varepsilon-{\rm fair}})=\mathcal{R}_{\lambda^{*(1)}, \lambda^{*(2)}}(g^*_{\varepsilon-{\rm fair}}) \leq \mathcal{R}_{\lambda^{*(1)}, \lambda^{*(2)}}(g) 
\leq R(g) + C \left(\mathcal{U}(g)-\varepsilon\right),
\end{equation*}
for all $g\in\mathcal{G}$ and for some constant $C >0$ that depends on $K$.
In the case of exact fairness ({\it e.g.} $\varepsilon = 0$) the following remark gives a specific characterization of the exact fair classifier.
\begin{rem}[Exact fairness]
    All previous results simplify in the exact fairness case setting where $\varepsilon = 0$. 
    Considering the reparametrization $\beta^*_k := \lambda^{*(1)}_k-\lambda^{*(2)}_k \in \mathbb{R}$, we deduce the optimal fair classifier in this case
    \begin{equation*}
        g_{{\rm fair}} ^{*}(x,s) \in \arg\max_{k} \left(\pi_s p_k(x,s)-s\beta^*_k\right), \;\, (x,s)\in\mathcal{X}\times\mathcal{S},
    \end{equation*}
    where  
    \begin{equation*}
        \beta^* \in \arg\min_{\beta \in \mathbb{R}^K} \sum_{s \in \mathcal{S}} \mathbb{E}_{X|S=s}\left[\max_k\left(\pi_s p_k(X,s)-s\beta_k\right)\right]\enspace .
    \end{equation*}
In view of Corollary~\ref{cor:UnfairnessCharac}, we have $\mathcal{U}(g_{{\rm fair}}^*) = 0$.
\end{rem}

\paragraph{Binary classification} Finally, we conclude this section with a particular focus on the binary classification setting where specific characterization of the optimal fair predictor can be obtained. 
\begin{coro}
\label{cor:UnfairnessShift}
Let $\varepsilon \geq 0$.
In the binary setting ($K=2$ with label space $\mathcal{Y} = \{0,1\}$),
the fairness constraint reduces to a single condition and 
the optimal fair classifier simplifies as
\begin{equation*}
g_{{\rm fair}} ^{*}(x,s) = \one_{\{ p_1(x,s) \geq \frac{1}{2} + \frac{s \beta^*}{2 \pi_s} \} }, \quad (x,s)\in\mathcal{X}\times\mathcal{S} \enspace,
\end{equation*} 
where, with the notation $F_s(t) = \mathbb{P}\left( p_1(X,S) \leq t  \; |  \; S=s \right)$ we have
\begin{enumerate}
\item[$i)$]$ \beta^* =0$ if $\left| F_1\left(\frac{1}{2}\right) - F_{-1} \left(\frac{1}{2} \right)\right| \leq \varepsilon$;
\item[$ii)$] $\beta^*$ is solution in $\beta$ of $\left| F_1\left(\frac{ \beta + \pi_1}{2 \pi_1}\right) = F_{-1} \left(\frac{ - \beta + \pi_{-1}}{2 \pi_{-1}} \right)\right| = \varepsilon$ otherwise.
\end{enumerate}
\end{coro}
The proof of this result follows directly from Theorem~\ref{thm:equivalenceApprox} by considering classifiers $g$ that satisfy the fairness constraint $ \left\vert \mathbb{P}\left(g(X,S) = 1  \ | \ S=-1 \right) - \mathbb{P}\left(g(X,S) = 1  \ | \ S= 1\right) \right\vert \leq \varepsilon$. This constraint ensures that the condition is also satisfied for $g(X,S) = 0$ since $g$ is a binary function.

The above result highlights several important facts about the characterization of the optimal fair classifier in the binary setting. First, the optimal rule is deduced just by thresholding the conditional probability $p_1$. The thresholding is not at the classical level $1/2$ ({\it e.g.} without fairness constraint) but at a shifting of this value by $\frac{s \beta^*}{2 \pi_s}$ to enforce fairness. 
Second, observe that the rule only depends on $p_1$ (and not $p_0$) for the same reason as in classical binary classification, that is $p_0 = 1-p_1$. This yields to a reduction of the number of Lagrange parameters
into a single one $\beta^*$. 
Notice that the case $\beta^* = 0$ means that the Bayes rule is already fair and then coincides with the $\varepsilon$-fair optimal predictor. In contrast, if $\beta^* \neq 0$, the optimal $\varepsilon$-fair rule differs from the Bayes rule and the modification of the rule is deduced by shifting the conditional probability.

\section{Data-driven procedure}
\label{sec:datadriven}

This section is devoted to the definition and the theoretical study of our empirical procedure
that relies on the {\it plug-in} principle.
The construction of our estimator is formally presented in Section~\ref{subsec:plug-in}
while its statistical properties are provided in~Section~\ref{subsec:consistency}.

\subsection{Plug-in estimator}
\label{subsec:plug-in}

The enhanced estimation procedure is in two steps. According to the definition of the optimal $\varepsilon$-fair predictor given in Theorem~\ref{thm:equivalenceApprox}, we first build estimators of the conditional probabilities $({p}_k)_k$ and then proceed with the estimation of the parameters $\lambda^*$ and $(\pi_s)_{s \in \mathcal{S}}$.
Notably, our data-driven procedure is semi-supervised as it relies on two independent datasets, one labeled and another unlabeled.

The first \emph{labeled} dataset $\mathcal{D}_n =  (X_i,S_i,Y_i)_{i= 1,\ldots ,n}$ contains \emph{i.i.d.} samples from the distribution $\mathbb{P}$. It allows to train estimators $(\hat{p}_k)_k$ of the conditional probabilities $({p}_k)_k$ by the means of any machine learning supervised algorithm, \emph{e.g.,} Random Forest, SVM. At this level, it is important to stress a key feature of the algorithm. Ones the empirical conditional probabilities $\hat{p }_k$ are trained, the theoretical analysis of the risk and the unfairness of the plug-in rule requires continuity conditions on the random variables $\hat{p}_k(X,S)$ (conditional on the learning sample, see Assumption~\ref{ass:continuity}). 
Notably, this is automatically satisfied whenever perturbing $(\hat{p}_k)_k$ with a continuous random noise (with a small magnitude to avoid deflating the statistical properties of the estimate). We insure such a property simply by randomization.
Indeed, let $u$ be a non negative real number. For each $k \in [K]$, we introduce
\begin{equation*}
\bar{p}_k(X,S,\zeta_k) := \hat{p}_k(X,S) + \zeta_k,
\end{equation*}
with $(\zeta_k)_{k \in[K]}$ being \emph{i.i.d.} according to a uniform distribution on $[0,u]$. This perturbation improves the fairness calibration in both theory and practice due to the fact that atoms for the random variables $\hat{p}_k(X,S)-\hat{p}_j(X,S)$ are avoided in this case.

The second \emph{unlabeled} dataset $\mathcal{D}'_N$ contains $N$ \emph{i.i.d.} copies of $(X,S)$. It is used to calibrate fairness. For $s\in\mathcal{S}$, the number of observations corresponding to $S=s$ is denoted by $N_s$, so that $N_{-1}+N_{1} = N$. On the one hand, the feature vectors in $\mathcal{D}'_N$ are denoted by $X_1^s, \ldots , X_{N_s}^s$ and are \emph{i.i.d.} data from the distribution $\mathbb{P}_{X^s}$ of $X|S=s$. On the other hand, the sensitive features from 
$\mathcal{D}'_N$ are denotes by $(S_1, \ldots, S_N)$.
The latter are \emph{i.i.d.} and are 
used to compute empirical frequencies $(\hat{\pi}_s)_{s\in\mathcal{S}}$ as estimates of $(\pi_s)_{s\in\mathcal{S}}$ (recall that $\pi_s = \mathbb{P}(S=s)$). Now notice that the estimation of parameters $(\lambda^{*(1)}, \lambda^{*(2)})$ only involves marginal distributions of $\mathbb{P}_{X|S=s}$ and $\mathbb{P}_{S}$. Therefore, this estimation part relies on the estimators $\hat{\pi}_s$, on the feature vectors $(X_1^s, \ldots , X_{N_s}^s)$, and on independent copies $(\zeta_{k,i}^s)_{k \in [K], i \in [Ns]}$ of a Uniform distribution on $[0,u]$ (for $s \in \mathcal{S}$).
In particular, we define $(\hat{\lambda}^{(1)},\hat{\lambda}^{(2)})$ as a minimizer over $\mathbb{R}_{+}^{2K}$ of $ \hat{H} (\lambda^{(1)} , \lambda^{(2)}) $ that is defined by (see the population counterpart given in Theorem~\ref{thm:equivalenceApprox})
    \begin{multline}
\label{eq:LamPlugInEps}
\hat{H}(\lambda^{(1)}, \lambda^{(2)})  := \sum_{s \in \mathcal{S}}
\frac{1}{N_s} \sum_{i=1}^{N_s}
\left[\max_k\left(\hat{\pi}_s \bar{p}_k(X_i^s,s,\zeta_{k,i}^s)-s(\lambda^{(1)}_k-\lambda^{(2)}_k)\right)\right] 
+ \varepsilon \sum_{k=1}^K (\lambda^{(1)}_k+\lambda^{(2)}_k)\enspace .
\end{multline}
Finally, our \emph{randomized} fair algorithm $ \hat{g}$ is defined as
\begin{equation}
\label{eq:eqPlugInEps}
\hat{g}_{\varepsilon}(x,s) = \arg\max_{k \in [K]}  \left(\hat{\pi}_s \bar{p}_k(x,s , \zeta_k )-s(\hat{\lambda}^{(1)}_k-\hat{\lambda}^{(2)}_k)\right)\enspace , \qquad (x,s)\in {\cal{X}} \times {\cal{S}}\;,
\end{equation}
Note that the construction of the plug-in rule $\hat{g}$ relies on $(x,s)$ but also on the perturbations $\zeta_k$ and $\zeta_{k,i}^s$ for $k\in [K]$, $i \in [N_s]$, and $s\in \mathcal{S}$, that are easily collected as \emph{i.i.d.} uniform random variables.
\begin{rem}
\label{rq:splitting}
    Classical datasets often contain only labeled samples. Then, our approach requires to split the data into two independent samples $\mathcal{D}_n $ and $\mathcal{D}'_N $, by removing labels in the latter.
    As illustrated in Section~\ref{subsec:Evaluation}, this splitting step is important to get the right level of fairness.
\end{rem}

\subsection{Statistical guarantees}
\label{subsec:consistency}
We are now in position to derive fairness and risk guarantees of our plug-in procedure. We need the following additional notation:
$\pi_{\min} :=\min_{s \in \mathcal{S}} \pi_s$ 
and $N_{\min} = \min(N_1 ; N_{-1})$.

\subsubsection{Universal fairness guarantee}

We first focus on fairness assessment and prove that the plug-in estimator $\hat{g}$ is asymptotically $\varepsilon$-fair, that is, it satisfies the requirement of Definition~\ref{def:espUnfairnessMeasure}. This control on the fairness will be established both in expectation and with high probability. In addition, we prove that the convergence rate of the unfairness to zero is parametric with the number of unlabeled data $N$. Notably, the fairness guarantee is distribution-free and holds for any estimators of the conditional probabilities.
\begin{theo}
\label{thm:unfairness} 
Let $\varepsilon \geq 0$.
There exists a constant $C > 0$ depending only on $K$ and $\pi_{\min}$ such that, for any estimators $\hat{p}_k$ of the conditional probabilities,  we have 
\begin{equation*}
\mathbb{E}\left[\mathcal{U}(\hat{g}_{\varepsilon})\right] \leq \varepsilon + \frac{C}{\sqrt{N}} \enspace.
\end{equation*}
\end{theo}
This first finite-sample bound on the fairness illustrates a key feature of our post-processing approach. It makes (asymptotically) $\varepsilon$-fair any off-the-shelf (unconstrained/unfair) estimators of the conditional probabilities. This post-processing step is especially appealing when the cost of re-training an existing learning algorithm is high.
While the former result provides a control of the unfairness on our algorithm in expectation, it is also appealing to have a thinner analysis of the unfairness through a high probability control.
\begin{theo}
\label{thm:caracEpsfairEstimator}
Let $0 < \delta < 1$ and define $C_{\delta} = 4
 K \sqrt{ 2  \log(\frac{4K}{\delta}) }
 $. Assume that $\varepsilon > \frac{\sqrt{2}C_{\delta}}{\sqrt{\pi_{\min}N}}$
and that $N \geq 2\dfrac{\log(1/\delta)}{\pi^2_{\min}}$. Then there exists an event $\mathcal{A}(\delta)$ that holds with probability $1-(K+2)\delta$ on which we have
\begin{equation*}
\dfrac{C_{\delta}}{\sqrt{N_{\min}}} < \varepsilon, \quad {\rm and} \;\; \forall k \in [K], \;\; \hat{\lambda}^{(1)}_k\hat{\lambda}^{(2)}_k = 0.
\end{equation*}
Besides on $\mathcal{A}(\delta)$, the following holds
\begin{itemize}
\item[1)] either
$\left|\mathcal{U}(\hat{g}_{\varepsilon}) -  \varepsilon \right| \leq  \frac{C_{\delta}}{\sqrt{N_{\min}}}$;
\item[2)] or 
$\mathcal{U}(\hat{g}_{\varepsilon}) <  \varepsilon - \frac{C_{\delta}}{\sqrt{N_{\min}}}$, and then we have $\hat{g} = \hat{g}_\varepsilon$ (for each $k \in [K]$, $\hat{\lambda}^{(1)}_k = \hat{\lambda}^{(2)}_k = 0$).
\end{itemize}
\end{theo}
This result has several levels of understanding. It highlights that the bound on the unfairness established in Theorem~\ref{thm:unfairness} is also valid with high probability, that is, there exists some constant $C>0$ such that 
$\mathcal{U}(\hat{g}_{\varepsilon})\leq \varepsilon + \frac{C}{\sqrt{ N_{\min}}} $ with high probability.
However, this result covers two  significantly different situations for $\hat{g}_\varepsilon$: the first case is when the unfairness of $\hat{g}_{\varepsilon}$ is small {\it w.r.t.} to $\varepsilon$. This means that the unconstrained classifier $ \hat{g}$ is already $\varepsilon$-fair and the action of the fairness constraint on our prediction function is null. In this case, we have $\hat{g}_{\varepsilon}= \hat{g}$. 
The second case, which is also the most expected one, is when at least one coordinate of the Lagrangian is non zero ({\it e.g.} either $\hat{\lambda}^{(1)}_k$ or $ \hat{\lambda}^{(2)}_k $ is non zero for some~$k$). Here, imposing the fairness constraint is relevant and the unfairness of $\hat{g}_{\varepsilon}$ falls within a small interval around $\varepsilon$.

From another perspective, all these conclusions are valid under some conditions on the desired level of unfairness $\varepsilon$ and the sample size $N$.
It is assumed that $N$ is large enough to make the fairness constraint meaningful.
However, it could be interesting to consider the case where $\varepsilon$ is smaller than the rate $\frac{1}{\sqrt{\pi_{\min} N}}$. (Observe that $\pi_{\min} N$ is the expectation of $N_{\min}$.)
In this case, our statements shows that all values of $\varepsilon \in [0, \frac{1}{\sqrt{\pi_{\min} N}}]$
lead, from the theoretical perspective, to the same bound on the unfairness of the resulting classifier.

\subsubsection{Consistency result}

In this part, we provide a control on the misclassification risk of $\hat{g}_{\varepsilon}$. 
Let us define the $\ell_1$-norm in $\mathbb{R}^K$ between the estimator $\mathbf{\hat{p}} := (\hat{p}_1,\ldots,\hat{p}_K) $ and the vector of the conditional probabilities $\mathbf{p} := (p_1,\ldots,p_K)$ by
 $
\|\mathbf{\hat{p}}(X,S)- \mathbf{p}(X,S)\|_1 = \sum_{k\in[K]} |\hat{p}_k(X,S) - p_k(X,S)|.$
We then derive the following bound.
\begin{theo}
\label{thm:excessRisk}
Let Assumption~\ref{ass:continuity} be satisfied. Assume that $\dfrac{N}{\log(N)} \geq 2 \pi_{\min}^{-2}$, then it holds that
\begin{equation*}
\begin{aligned}
 \mathbb{E}  [\mathcal{R}_{\lambda^{*(1)},\lambda^{*(2)}}(\hat{g}_{\varepsilon}) ] - \mathcal{R}_{\lambda^{*(1)},\lambda^{*(2)}}(g^*_{\varepsilon-{\rm fair}}) 
\leq 
C\left(\mathbb{E}\left[\|\mathbf{\hat{p}}(X,S)- \mathbf{p}(X,S)\|_1\right] + \sum_{s \in \mathcal{S}} \mathbb{E}\left[ |\hat{\pi}_s-\pi_s| \right]  + \dfrac{\log(N)}{\sqrt{N}}+ u \right)   \enspace ,
\end{aligned}
\end{equation*}
where $C > 0$ depends on $K$ and $\pi_{\min}$.
\end{theo}
This result highlights that the excess fair-risk of $\hat{g}$ depends on 
i) the $L_1$-risk of $\mathbf{\hat{p}}$ for estimating the conditional probabilities; 
ii) the efficiency of the estimators  $(\hat\pi_s)_{s\in\mathcal{S}}$; 
iii) a bound on the unfairness of the classifier; and 
iv) the upper-bound $u$ on the regularizing perturbations. 
In view of Theorem~\ref{thm:caracEpsfairEstimator}, 
$\hat{g}_{\varepsilon}$ is consistent {\it w.r.t.} the misclassification risk as soon as the estimator $\mathbf{\hat{p}}$ is consistent in $L_1$-norm.
In particular, we can establish the following result.
\begin{coro}
\label{corr:consist}
Let $\varepsilon \geq 0$, if $\mathbb{E}\left[\|\mathbf{\hat{p}}(X,S)- \mathbf{p}(X,S)\|_1\right]  \to 0$ and $u = u_{n}  \to 0$ when $n \to \infty$, we have
\begin{equation*}
 \vert \mathbb{E}\left[\mathcal{R}(\hat{g}_{\varepsilon})\right] -  \mathcal{R}(g^*_{\varepsilon-{\rm fair}}) \vert
 \to 0, \qquad \text{as}\quad  n,N  \to \infty   \enspace.
\end{equation*}
\end{coro}
Theorem~\ref{thm:unfairness} and Corrolary~\ref{corr:consist} directly imply that $\hat{g}_{\varepsilon}$ performs asymptotically as well as $g^*_{\varepsilon-{\rm fair}}$ both in terms of  fairness and accuracy provided that the estimators of $p_k$ are consistent {\it w.r.t.} the $L_1$ risk.

\subsubsection{Rates of convergence}

This section is dedicated to the study of rates of convergence {\it w.r.t} the excess fair-risk.
To this end, we require additional assumptions on the regression functions $p_k$.
\begin{ass}(Smoothness assumption)
\label{ass:smoothness}
For all $k \in [K]$,  the regression function $p_k$ is Lipschitz.
\end{ass}
The bound on the excess-risk provided in Theorem~\ref{thm:excessRisk}
depends on $\mathbb{E}\left[\|\mathbf{\hat{p}}(X,S)- \mathbf{p} (X,S)\|_1\right] $. Imposing additional regularity constraint on $\mathbf{p}$, this term can further be controlled. 
For instance, if we assume that for each $k \in [K]$, the regression functions $p_k$ are Lipschitz then well established nonparametric estimators of $p_k$, such as local polynomials or kernel based methods, lead to
\begin{equation*}
\mathbb{E}\left[\|\mathbf{\hat{p}}(X,S)- \mathbf{p}(X,S)\|_1\right] 
\leq
Cn^{-1/(2+d)}\enspace. 
\end{equation*}
In this case a straightforward consequence of Theorem~\ref{thm:excessRisk} is that for $u\leq n^{-1/(2+d)}$
\begin{equation*}
\begin{aligned}
 \mathbb{E}  [\mathcal{R}_{\lambda^{*(1)},\lambda^{*(2)}}(\hat{g}_{\varepsilon}) ] - \mathcal{R}_{\lambda^{*(1)},\lambda^{*(2)}}(g^*_{\varepsilon-{\rm fair}}) 
\leq 
C\left(n^{-1/(2+d)} \bigvee N^{-1/2}\right)   \enspace .
\end{aligned}
\end{equation*}
In particular, if $N$ is sufficiently large, that is $N^{-1/2} =  \smallO\left( n^{-1/(2+d)} \right)$, the obtained rates is of the same order as the minimax rates in classification setting without fairness constraint~\cite{Audibert_Tsybakov07}.
Interestingly, it is possible to obtain faster rates under a stronger assumption than Assumption~\ref{ass:continuity}.
\begin{ass}(Density assumption)
\label{ass:denistyAss}
For any $k,j  \in[K]$ and $s\in \mathcal{S}$, we assume that conditional on $S = s$, the random variable
$p_k(X,S)-p_j(X,S)$ admits a bounded density.
\end{ass}
Note that under Assumption~\ref{ass:denistyAss}, the Tsybakov's margin condition is satisfied with parameter $\alpha = 1$. 
Taking advantage of the margin condition, we can establish the following result.
\begin{theo}
\label{thm:fastRates}
For $\varepsilon > 0$ and for a sample size $N$ such that $\dfrac{N}{\log(N)} \geq 2\pi_{\min}^{-2}$,
the following holds
\begin{equation*}
\begin{aligned}
 \mathbb{E}  [\mathcal{R}_{\lambda^{*(1)},\lambda^{*(2)}}(\hat{g}_{\varepsilon}) ] - \mathcal{R}_{\lambda^{*(1)},\lambda^{*(2)}}(g^*_{\varepsilon-{\rm fair}}) 
\leq 
C\left(\mathbb{E}\left[\|\mathbf{\hat{p}}(X,S)- \mathbf{p}(X,S)\|^2_{\infty}\right] +  \dfrac{\log^2(N)}{N}+ u^2\right)   \enspace ,
\end{aligned}
\end{equation*}
where $C > 0$ depends on $K$ and $\pi_{\min}$. 
\end{theo}
The major consequence of the above result is that fast rates of convergence (faster than $n^{-1/2}$) can be obtained for the excess fair-risk. Specifically, 
if under Assumption~\ref{ass:smoothness}, the estimator satisfies
\begin{equation}
\label{eq:eqCveinfty}
\mathbb{E}\left[\left\|\hat{\bf{p}}(X,S)-{\bf p}(X,S) \right\|_{\infty}\right] \leq C\log(n)n^{-1/(2+d)}\enspace, 
\end{equation}
(which is again the case for popular methods) under Assumption~\ref{ass:denistyAss}, it holds that
\begin{equation*}
\begin{aligned}
 \mathbb{E}  [\mathcal{R}_{\lambda^{*(1)},\lambda^{*(2)}}(\hat{g}_{\varepsilon}) ] - \mathcal{R}_{\lambda^{*(1)},\lambda^{*(2)}}(g^*_{\varepsilon-{\rm fair}}) 
\leq 
C\log^2(n) \left(n^{-2/(2+d)} \bigvee N^{-1} \right)   \enspace .
\end{aligned}
\end{equation*}
Interestingly, if the size of the unlabeled sample $N$ is sufficiently large ($N \geq \log(n)^{-2}n^{2/(2+d)}$), then up to a logarithmic factor the established rates of convergence is of the same order as the minimax {\it fast} rates of convergence for plug-in classifiers (see~\cite{Audibert_Tsybakov07}) in supervised classification without fairness constraint.   
Hence, we manage to show that fast rates can be also achieved in the algorithmic fairness framework under Margin type assumption. Note that the condition required in Equation~\eqref{eq:eqCveinfty} is, for instance fulfilled by local polynomial estimator or $k$NN classifiers under Assumption~\ref{ass:smoothness}.
Finally, we also want to point out that we restrict our analysis to the case where the regression functions $p_k$ are Lipschitz to ease the presentation. 
However, we can extend our results to the case where the regression functions are in a H\"older class. 

\section{Numerical Evaluation}
\label{sec:Evaluation}

We now evaluate our method numerically\footnote{The source of our method can be found at \url{https://github.com/curiousML/epsilon-fairness}.}. Section~\ref{subsec:Evaluation} illustrates the efficiency of the $\varepsilon$-fairness algorithm on synthetic data, while experiments on real datasets are provided in Section~\ref{subsec:RealData}.
Up to our knowledge, imposing the fairness constraint in multi-class classification in a model-agnostic post-processing approach is only addressed in~\citep{alghamdi2022beyond}. Therefore we will mainly compare our method to~\citep{alghamdi2022beyond} for multi-class tasks and to the state-of-the-art in-processing approach~\citep{agarwal2019fair} that is designed for binary tasks.

\subsection{Implementation of the algorithm}\label{sec:Implementation}

Let us focus on the implementation of the algorithm producing an $\varepsilon$-fairness classifier. Although the exact fairness setting allows for improvements using accelerated gradient descent, we do not focus on this point and simply identify the exact fair algorithm to the approximate fair one with $\varepsilon=0$.

The proposed approximate fair algorithm is defined in Eq.~\eqref{eq:eqPlugInEps} and requires to solve an optimization problem in Eq.~\eqref{eq:LamPlugInEps}. The implementation--pseudo-code is provided 
in Algorithm~\ref{alg:optimization}. 

\begin{algorithm}
   \caption{$\varepsilon$-fairness calibration}
   \label{alg:optimization}
\begin{algorithmic}
   \STATE {\bfseries Input:}  Approximate fairness parameter $\varepsilon$, new data point $({x}, {s})$, base estimators $(\bar{p}_k)_k$, unlabeled sample $\mathcal{D}'_N$, $(\zeta_{k})_k$ and  \emph{i.i.d} uniform perturbations $(\zeta_{k,i}^{s})_{k,i,s}$ in $[0,10^{-5}]$.
   \STATE {\bf \quad Step 0}. Split $\mathcal{D}'_N$ and construct the samples $(S_1,\ldots,S_N)$ and $\{X_1^s,\ldots,X_{N_s}^s \}$, for $s\in \mathcal{S}$;
   \STATE {\bf \quad Step 1.} Compute the empirical frequencies $(\hat{\pi}_s)_s$ based on $(S_1,\ldots,S_N)$;
   \STATE {\bf \quad Step 2.} 
   Compute $\hat{\lambda}^{(1)} = (\hat{\lambda}^{(1)}_1,\ldots, \hat{\lambda}^{(1)}_K)$ and $\hat{\lambda}^{(2)} = (\hat{\lambda}^{(2)}_1,\ldots, \hat{\lambda}^{(2)}_K)$ as a solution of Eq.~\eqref{eq:LamPlugInEps};  \\  \hfill \texttt{Sequential quadratic programming of Section~\ref{sec:Implementation} can be used for this step.} 
   \STATE {\bf \quad Step 3.} Compute $\hat{g}$ thanks to Eq.~\eqref{eq:eqPlugInEps};
   \STATE {\bfseries Output:} $\varepsilon$-fair classification $\hat{g}(x,s)$ at point $(x,s)$. 
\end{algorithmic}
\end{algorithm}
First of all, base estimators $ (\bar{p}_k)_k$ are needed as inputs of the algorithm.
We emphasize that we can fit any off-the-shelf estimators on the labeled dataset $\mathcal{D}_n$. In particular, one can use efficient ML algorithms that are already pre-trained and that are eventually expensive to re-train. This is one of the main advantages of post-training approaches over in-processing ones. In addition, randomization in the definition of $\bar{p}_k$ provides good theoretical properties for fairness calibration (\emph{c.f.} Section~\ref{subsec:consistency}).

Once $ (\bar{p}_k)_k$ are computed, the fair classifier $\hat{g}$ relies on the estimators $\hat \lambda^{(1)}$ and $\hat\lambda^{(2)}$ computed in {\bf Step~2.} of the algorithm. This requires solving the minimization problem in Equation~\eqref{eq:LamPlugInEps}. The corresponding objective function is convex but non-smooth due to the evaluation of the \textit{max} function. We regularize the objective function by replacing the hard-max by a soft-max. 
Namely, for $\beta$ a positive real number designating the temperature parameter and $a = (a_1,\ldots,a_K)^\top \in \mathbb{R}^K$, we set
\begin{eqnarray*}
{\rm{softmax}}(a)  :=  \sum_{k=1}^{K}\sigma_\beta(a)_k \cdot a_k , \quad
 {\rm where} &  \sigma_\beta(a)_k := \frac{\exp\left( a_k/\beta \right)}{\sum_{k=1}^{K} \exp\left( a_k/\beta \right)}\enspace.
\end{eqnarray*}
Whenever $\beta \to 0$, the soft-max reduces to the max function. Problem~\eqref{eq:LamPlugInEps} with the soft-max relaxation is smooth enough to be solved by a constrained optimization method, such as sequential quadratic programming \citep{Fu2019SequentialQP, nie2007sequential}. Empirical study shows that $\beta = 0.005$ enables a good accuracy of the algorithm, without deviating too much from the original solution. 

Instead of regularizing the objective function, one can alternatively use sampling methods such as cross-entropy optimization \citep{rubinstein1999cross} on the original objective function. Despite their precision, the downside of this method is the induced computational complexity, that grows much faster with the dimension than the complexity induced by smoothing techniques. Hence, the regularization approach has been preferred in the following numerical study.

\subsection{Evaluation on synthetic data}
\label{subsec:Evaluation}

Before illustrating our method on real datasets, we choose to evaluate our methodology on synthetic data, in order to better understand its performance.

\subsubsection{Synthetic data} 
\label{subsubsec:SyntheticData}

Let us define the synthetic data $(X, S, Y)$. For all $ k\in [K]$ we set $\mathbb{P}(Y = k) = 1/K$. Conditional on $Y = k$, features $X\in \mathbf{R}^d$ follow a Gaussian mixture of $m$ components:  
$$(X | Y  =  k)  \sim   \frac{1}{m}\sum_{i=1}^{m}\mathcal{N}_d(c^k + \mu_i^k, I_d)$$ 
with $c^k \sim \mathcal{U}_d(-1, 1)$, and 
$\mu_1^k, \dots, \mu_m^k \sim \mathcal{N}_d(0, I_d);$ 
while the sensitive feature $S\in \{-1, +1\}$ follows a Bernoulli {\it contamination} with parameter $p$ or $1-p$ depending on $k$: 
$$
(S | Y = k ) \sim   2 \cdot \mathcal{B}(p) -1 \quad \text{ if } k \leq \left\lfloor K/2\right\rfloor
\quad\quad \text{and} \quad\quad
(S | Y = k )
  \sim  2 \cdot \mathcal{B}(1-p) -1 \quad \text{if } k > \left\lfloor K/2\right\rfloor.$$
From this model, we can deduce an expression of the Bayes classifier $g^*$. Indeed for each $k \in [K]$, since conditional on $Y=k$, the random variables $X$ and $S$ are independent and $\mathbb{P}(Y=k) = 1/K$, we have from the Bayes formula
\begin{equation*}
p_k(x,s) = \dfrac{f_{X|Y=k}(x) \mathbb{P}(S=s|Y=k)}{\sum_{j=1}^K f_{X|Y=j}(x) \mathbb{P}(S=s|Y=j)}\enspace ,    
\end{equation*}
where $f_{X|Y=k}$ is the density of $X$ conditional on $Y=k$.
In view of the expression of the conditional probabilities $p_k$, the Bayes classifier $g^*$ can be expressed as
\begin{equation*}
 g^* (x,s) \in \arg\max_{k \in [K]} f_{X|Y=k}(x) \mathbb{P}(S=s|Y=k) \enspace .   
\end{equation*}
We exploit the above formula to evaluate the unfairness of $g^*$ {\it w.r.t.} the parameter $p$. Figure~\ref{fig:figUnfairBayes} displays the obtained results. Interestingly,
we see that parameter $p$ measures the historical bias in the dataset.
Hence, this synthetic data structure enables to challenge different aspects of the algorithm. 
In particular, the data becomes fair when $p = 0.5$ and completely unfair when $p \in \{0,1\}$ (see also Figure~\ref{fig:syn_data} in Appendix~\ref{sec:add_numres} for an illustration).
As default parameters, we set $K = 6$, $p=0.75$, $m = 10$, and $d = 20$.

\begin{figure}
\begin{center}
\includegraphics[scale=0.5]{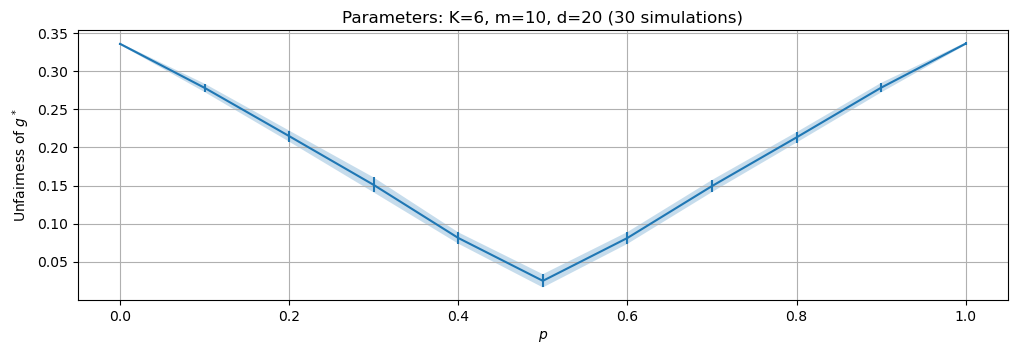}
\caption{Unfairness of the Bayes classifier $g^*$ w.r.t. parameter $p$. We report the means and standard deviations over 30 simulations.}
\label{fig:figUnfairBayes}
\end{center}
\end{figure}

\subsubsection{Simulation scheme} 

We compare our method to the unfair approach. We set $u=10^{-5}$ and estimate the conditional probabilities $p_k$ by Random Forest (RF) with default parameters in \texttt{scikit-learn}. We generate $n = 5000$ synthetic examples and split the data into three sets ($60\%$ training, $20\%$ hold-out and $20\%$ unlabeled). 

The performance of a classifier $g$ is evaluated by its empirical accuracy ${\rm Acc}(g)$ on the hold-out set $\mathcal{T}$
\begin{equation*}
    Acc(g) = \dfrac{1}{\left|\mathcal{T}\right|}\sum\limits_{(X, S, Y) \in \mathcal{T}} \mathds{1}_{\{g(X, S) = Y\}}\enspace.
\end{equation*}
The unfairness of $g$ is measured on the hold-out set by the empirical counterpart $\hat{\mathcal{U}}(g)$ of the unfairness given in Definition~\ref{def:espUnfairnessMeasure}, that is,
\begin{equation*}
    \hat{\mathcal{U}}(g) = \max_{k \in[K]} \left|
    \hat \nu_{g|-1} (k) - \hat \nu_{g|1} (k)\right|\enspace,
\end{equation*}
where $\hat \nu_{g|s} (k)= \dfrac{1}{|\mathcal{T}^{s}|}\sum\limits_{(X, S, Y) \in \mathcal{T}^{s}}\mathds{1}_{\{g(X, S) = k\}}$ is the empirical distribution of $g(X, S) | S = s$ on the conditional hold-out test $\mathcal{T}^s = \left\{ (X, S, Y)\in \mathcal{T} \; | \; S = s \right\}$.

\subsubsection{Fairness versus Accuracy}

\begin{figure}
\begin{center}
\includegraphics[scale=0.63]{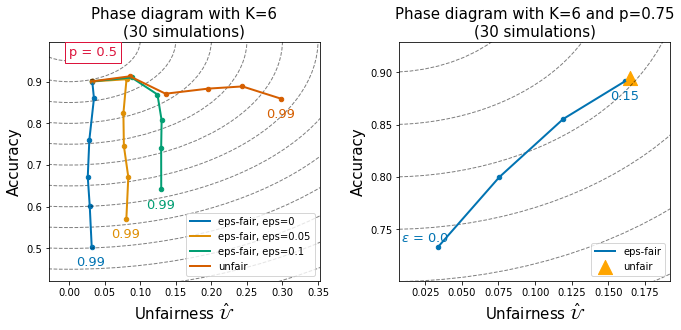}
\caption{(Accuracy, Unfairness) phase diagrams \emph{w.r.t.} \textit{Left} the level of bias $p$ between $0.5$ and $0.99$; \textit{Right} the accuracy-fairness trade-off  parameter $\varepsilon$. Top-left corner gives the best trade-off.}
\label{fig:syn_phase_diagram2}
\end{center}
\end{figure}
Figure~\ref{fig:syn_phase_diagram2}-Left illustrates how fairness and accuracy vary across different levels of unfairness, quantified by $p\in\{0.5, 0.6, 0.7, 0.8, 0.9, 0.99\}$, in both the unfair and fair random forests with $\varepsilon\in\{0, 0.05, 0.1\}$.
Figure~\ref{fig:syn_phase_diagram2}-Right presents the fairness and accuracy of our $\varepsilon$-fairness method for $\varepsilon \in\{0, 0.05, 0.1, 0.15\}$.
Note that the performance evolves as expected: enforcing fairness degrades the accuracy and the trade-off accuracy-fairness is controlled by the parameter $\varepsilon$.
From Figure~\ref{fig:syn_phase_diagram2}-Right, for exact fairness ($\varepsilon=0$), the gain in fairness is particularly salient and effective.
By contrast, whenever $\varepsilon=0.15$, the fair classifier becomes similar to the unfair method, confirming the result in Section~\ref{subsubsec:SyntheticData} that the original unfairness of the problem is around $\varepsilon = 0.15$.
From Figure~\ref{fig:syn_phase_diagram2}-Left, we additionally notice that: 1) the fairness efficiency of the algorithm is particularly significant for datasets with large historical bias ($p=0.9$ or $0.99$); 2) our method succeeds at reaching the required unfairness level up to small approximation terms (vertical curves  as soon as the unfairness bound $\varepsilon$ is reached); 3) as claimed in Theorem~\ref{thm:caracEpsfairEstimator}, when the unconstrained classifier is already $\varepsilon$-fair, the action of the fairness constraint on $\hat{g}_{\varepsilon-{\rm fair}}$ is null and we have $\hat{g}_{\varepsilon-{\rm fair}} = \hat{g}$ (horizontal parts of the curves).
We also illustrates in Figure~\ref{fig:synthetic_distribution} that the distribution of $\hat{g}_{\rm fair}$ is independent from $S$.

\begin{figure}
\begin{center}
\includegraphics[scale=0.46]{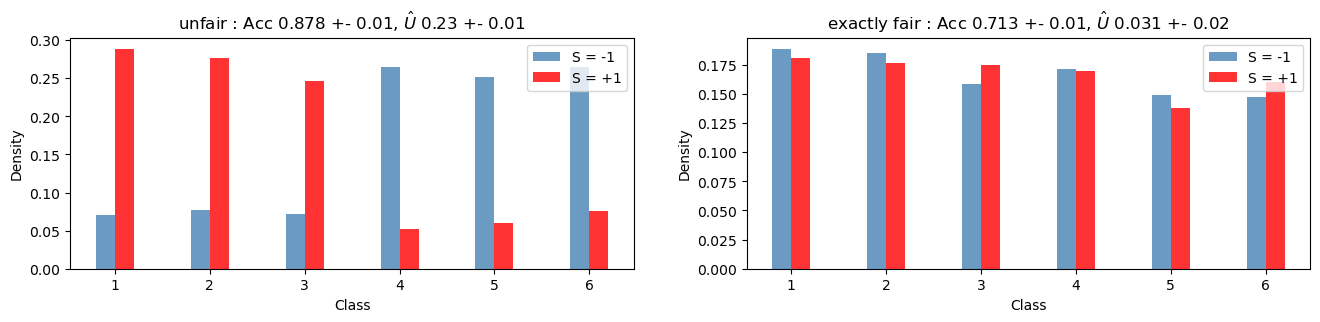}
\caption{Empirical distribution of $\hat{g}$ on 30 simulations. \textit{Left}: unfair classifier; \textit{Right}: exactly-fair classifier.}
\label{fig:synthetic_distribution}
\end{center}
\end{figure}

\paragraph*{Splitting the sample}
When an unlabeled dataset is not available, the samples $\mathcal{D}_n$ and $\mathcal{D}_N'$ follow from splitting the initial dataset, see Remark~\ref{rq:splitting}. Our theoretical study relies strongly on the independence between both datasets $\mathcal{D}_n$ and $\mathcal{D}_N'$. Figure~\ref{fig:splitting_rf} numerically illustrates the importance of such condition for the fairness but also the accuracy of our proposed method. Indeed, whenever the splitting is not performed (left parts of plots), the fairness performance of the fair algorithm may even be worse than the unfair method. This emphasize that splitting is crucial and enables to avoid over-fitting on the training set.

\begin{figure}
\begin{center}
\includegraphics[scale=0.46]{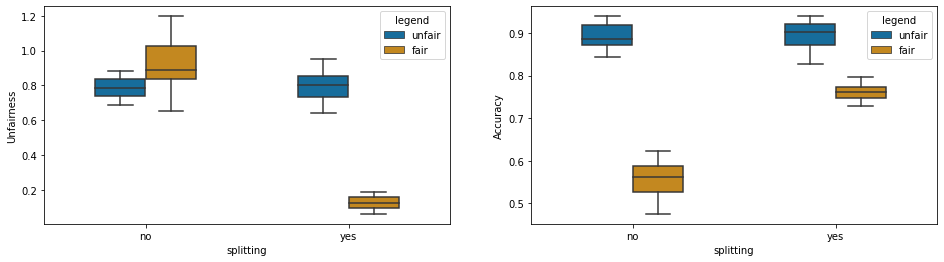}
\caption{Empirical impact of data splitting on unfairness (Left -- the lower the better) and accuracy (Right: accuracy -- the higher the better). Boxplots are generated over $30$ repetitions with $p = 0.75$. 
The non-splitting procedure involves
two sets ($80\%$ training and $20\%$ hold-out): in  this particular case we use the training set (instead of the unlabeled) to compute empirical frequencies $(\hat{\pi}_s)_{s\in\mathcal{S}}$.
}
\label{fig:splitting_rf}
\end{center}
\end{figure}

\subsection{Application to real datasets}
\label{subsec:RealData}

In this section, we illustrate the performance of our methodology on real data and compare it with a benchmark of three State of the Art algorithms \citep{Adversarial19,agarwal2019fair,alghamdi2022beyond}.

\subsubsection{Datasets}


The performance of the method is evaluated on two real datasets : DRUG and CRIME. Hereafter, we provide a short description of these datasets.

\paragraph*{Drug Consumption (DRUG)} This dataset~\cite{fehrman2017five} contains demographic information such as age, gender, and education level, as well as measures of personality traits thought to influence drug use for 1885 respondents. The task is to predict cannabis use, where the 7 levels of drug use have been simplified into $K=4$ categories (never used, not used in the past year, used in the past year, and used in the past day) for multi-class outcomes or $K=2$ categories (used or not used in the past year) for binary outcomes. The binary sensitive feature is education level (college degree or not).

\paragraph*{Communities\&Crime (CRIME)} This dataset contains socio-economic, law enforcement, and crime data about communities in the US with 1994 examples. The task is to predict the number of violent crimes per $10^5$ population which, we divide into $K = 5$ (multi-class outcomes) or $K=2$ (binary outcomes) balanced classes based on equidistant quantiles. Following~\cite{calders2013controlling}, the sensitive feature is a binary variable that corresponds to the ethnicity.

\subsubsection{Methodology} 

We illustrate our $\varepsilon$-fair method\footnote{See \url{https://github.com/curiousML/epsilon-fairness}.} with linear and nonlinear multi-class classification methods. For linear models, we consider one-versus-all logistic regression (reglog); for nonlinear models, Random Forest (RF) and LightGBM (GBM). For reglog, we use the default parameters in scikit-learn.
For RF and GBM, we use a $3$-fold cross-validation random search to select the best hyperparameters with the training set:\\
$\bullet$ For RF, we set the number of trees in $\{10,11, \dots, 200\}$, the maximum depth of each tree in $\{2, 3,  \dots, 16\}$, the minimum number of samples required to split an internal node in $\{2, 3, \dots, 10\}$, and the minimum number of samples required to be at a leaf node in $\{1,\dots, 8\}$;\\
$\bullet$ For GBM, we set the $L1$ and $L2$ regularization term on weights both in $\{0, 0.1, 1, 2, 5, 10, 20, 50\}$, the number of boosted trees in $\{10, 11, \dots, 200\}$, the maximum tree leaves in $\{6, 7, \dots, 50\}$, the maximum depth of each tree in $\{2, 3, \dots, 16\}$, and the minimum number of samples required in a child node for a split to occur in the tree in $\{10, 11, \dots, 100\}$.

\begin{figure}
\begin{center}
\includegraphics[scale=0.4]{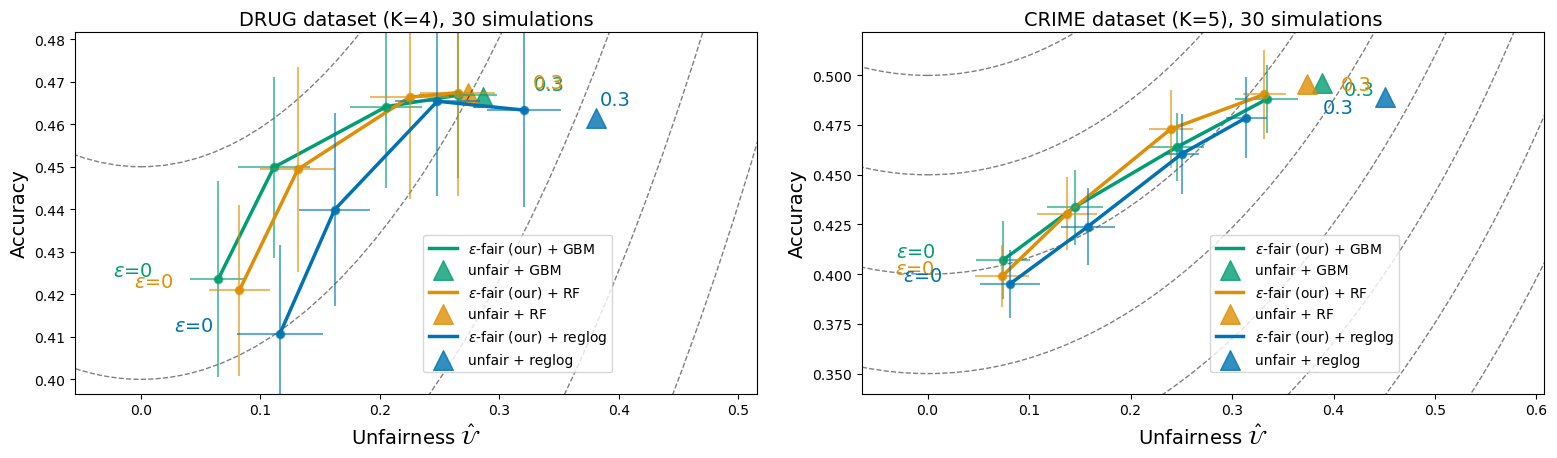}
\caption{(Accuracy, Unfairness) phase diagrams that shows the evolution, w.r.t. the accuracy-fairness trade-off parameter $\varepsilon \in [0, 0.1, 0.2, 0.3]$. We
report the means and standard deviations over the $30$ repetitions. Top-left corner gives the best trade-off. }
\label{fig:2multiclass_real}
\end{center}
\end{figure}

Note that the numerical experiments presented in Figure~\ref{fig:2multiclass_real} confirm our findings on synthetic data. Our method have good performance in term of unfairness while the accuracy slightly increases when the level $\varepsilon$ of desired fairness increases. Besides, the performance of the $\varepsilon$-fair classifier becomes closer to the base (unfair) when the fairness constraint is released.

\subsubsection{Benchmarks}

We aim at highlighting the numerical efficiency of our method in terms of accuracy-fairness trade-off curves. For this purpose, we compare our $\varepsilon$-fairness method to the following benchmarks :

\paragraph*{Fair-learn} For binary classification tasks, the current state-of-the-art is established by the in-processing approach \citep{agarwal2019fair}\footnote{The method in~\citep{agarwal2019fair} was developed for \emph{Equality of Odds} but the code is also implemented for \emph{Demographic Parity} see \url{https://github.com/fairlearn/fairlearn}.}. The authors present a reduction-based algorithm, which is an extention of the Fair-Lasso. The Fair-Lasso algorithm is a variant of the traditional Lasso algorithm that incorporates fairness constraints, aiming at finding a fair solution while maintaining good predictive performance. We use the following trade-off tolerances $[0.0001, 0.5, 1, 2.5, 5, 10]$.

\paragraph*{Fair-adversarial} 
The paper~\cite{Adversarial19}\footnote{We use IBM AIF360 library \url{https://aif360.readthedocs.io/en/stable/modules/algorithms.html}.} presents an in-processing method for reducing bias using adversarial training: a primary model, which is trained to perform a specific task, and a bias correction model, which is trained to reduce the bias in the primary model's predictions. Note that we cannot universally apply this method on any pre-trained classifier. We use a Neural Network (NN) as the base classifier and set the following parameters: \texttt{num\_epochs = 200}, \texttt{batch\_size = 128}, \texttt{classifier\_num\_hidden\_units = 50} (see the python package \texttt{AIF360}). We use the following trade-off tolerances $[0.01, 0.1, 0.5, 0.9, 1]$.

\paragraph*{Fair-projection} For multi-class classification tasks, we compare our result to the recent post-processing approach~\cite{alghamdi2022beyond}\footnote{The code can be found at \url{https://github.com/HsiangHsu/Fair-Projection}.}. The authors propose a method based on information projection by reweighting the outputs of a pre-trained classifier to satisfy specific group-fairness requirements. The trade-off tolerances are $[0, 0.1, 0.2, 0.5, 0.9]$.

\begin{figure}
\begin{center}
\includegraphics[scale=0.4]{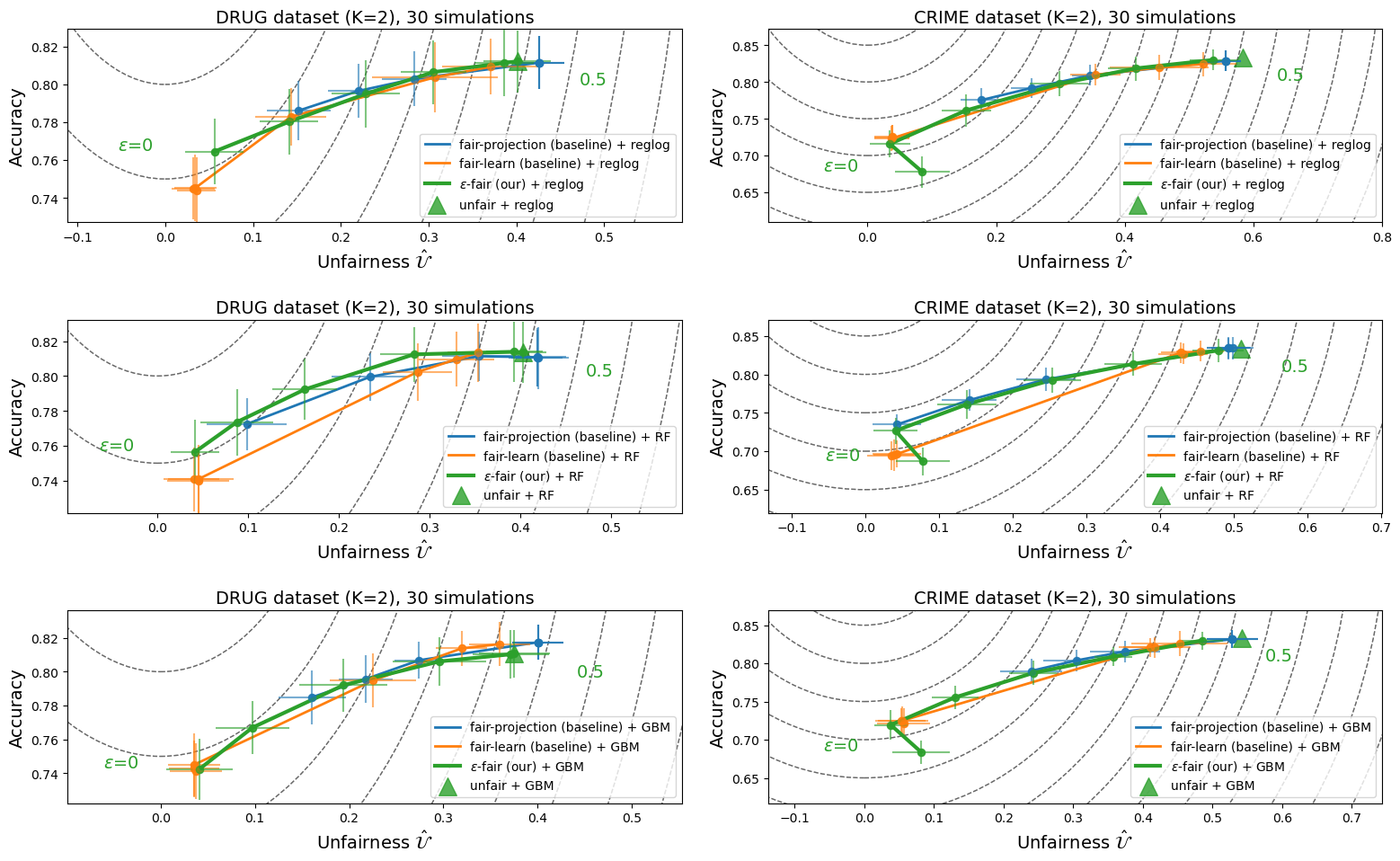}
\caption{(Accuracy, Unfairness) phase diagrams that shows the evolution, \emph{w.r.t.} the accuracy-fairness trade-off tolerances. We
report the means and standard deviations over $30$ repetitions. Top-left corner gives the best trade-off.}
\label{fig:2binary_real_baseline}
\end{center}
\end{figure}

\subsubsection{Results}

\begin{figure}
\begin{center}
\includegraphics[scale=0.6]{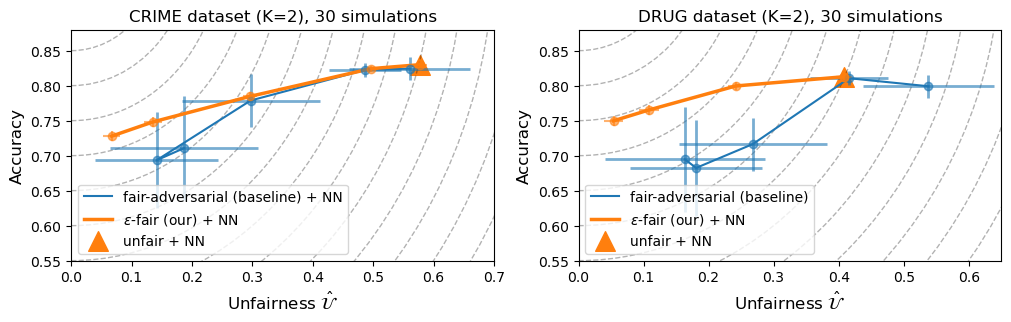}
\caption{(Accuracy, Unfairness) phase diagrams that shows the evolution, \emph{w.r.t.} the accuracy-fairness trade-off tolerances. For $\varepsilon$-fair classifier we vary $\varepsilon\in\{0.01, 0.1, 0.3, 0.5, 0.9\}$. We
report the means and standard deviations over $30$ repetitions. Top-left corner gives the best trade-off.}
\label{fig:binary_crime_adversary}
\end{center}
\end{figure}

\begin{figure}
\begin{center}
\includegraphics[scale=0.4]{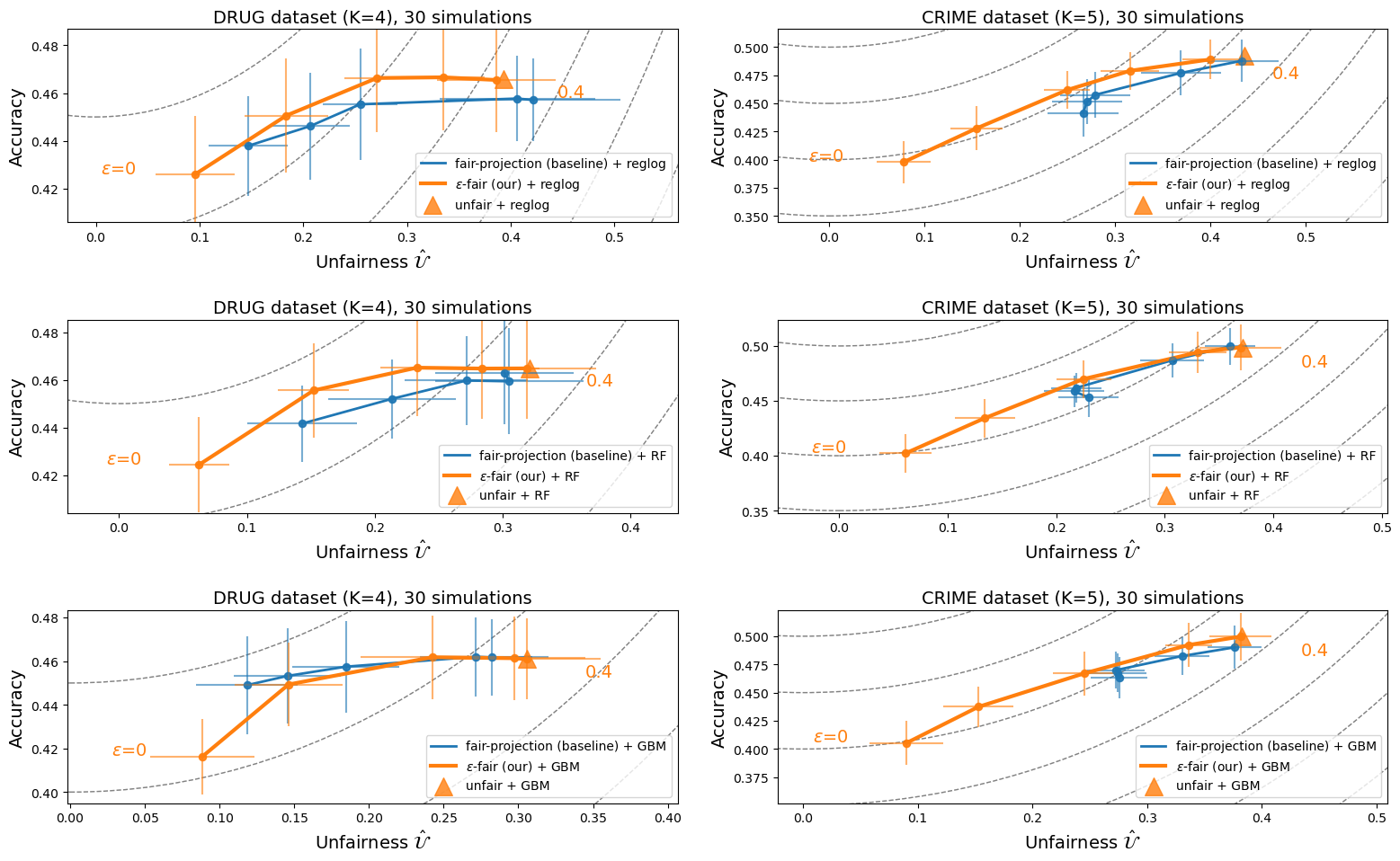}
\caption{(Accuracy, Unfairness) phase diagrams that shows the evolution, \emph{w.r.t.} the accuracy-fairness trade-off tolerances. We
report the means and standard deviations over $30$ repetitions. Top-left corner gives the best trade-off.}
\label{fig:2varmulticlass_real_baseline}
\end{center}
\end{figure}

\paragraph*{Performance in binary case ($K=2$)}
We analyze the efficiency of the $\varepsilon$-fairness method compared to \texttt{fair-learn}, \texttt{fair-projection} and \texttt{fair-adversarial} for binary classification.
Numerical experiments on DRUG and CRIME presented in Figure~\ref{fig:2binary_real_baseline} reveal that our method is very efficient in both accuracy and fairness and at least competitive (if not better) in several aspects :
\begin{enumerate}
    \item \textbf{Competitive fairness.} Overall, our $\varepsilon$-fair classifier outperforms \texttt{fair-projection} classifier in terms of exact fairness ($\varepsilon = 0$) and achieves similar performance as the state-of-the-art benchmark \texttt{fair-learn}.

    \item \textbf{Competitive accuracy.} Although we obtain similar accuracies using reglog and GBM, our algorithm seems more efficient than \texttt{fair-learn} using RF. Compared to \texttt{fair-projection} our algorithm is competitive in terms of accuracy for $\varepsilon \geq 0.1$ in both datasets.
\end{enumerate}
From Figure~\ref{fig:binary_crime_adversary}, our $\varepsilon$-fair predictor outperforms \texttt{fair-adversarial} predictor both in terms of accuracy and fairness.
Note that since \texttt{fair-learn} and \texttt{fair-adversarial} are in-processing methods their running time (using the dedicated package) is much higher than our algorithm.

\paragraph*{Performance in multi-class case ($K\geq3$).} 
We analyze the efficiency of the $\varepsilon$-fairness method compared to the baseline \texttt{fair-projection} for multi-class classification.
The numerical experiments are presented in Figure~\ref{fig:2varmulticlass_real_baseline}.
In multi-class tasks, empirical results highlight the efficiency of our approach to enforce fairness when $\varepsilon$ decreases. Indeed, our methodology achieves better fairness results under the DP constraint than \texttt{fair-projection} while maintaining competitive accuracy. Moreover, our fairness calibration is close to the pre-specified level, regardless of the base algorithm (reglog, RF or GBM).  

Finally, our methodology only use a portion of the dataset to train a classifier, while reserving the remaining portion as unlabeled. Despite using relatively small datasets, consisting of about 1000 examples, our approach performed better than other benchmark methods trained on full labeled datasets.

\section{Conclusion}
\label{sec:conclusion}

In the multi-class classification framework, we provide an optimal fair classification rule under DP constraint and derive misclassification and fairness guarantees of the associated plug-in fair classifier (see Algorithm~\ref{alg:optimization}). We handle both exact and approximate fairness settings and show that our approach achieves distribution-free fairness and can be applied on top of any probabilistic base estimator. We also establish rates of convergence for our procedure.
Up to our knowledge, the present contribution is the first statistical analysis in approximate fairness context. In particular, we consider here the multi-class setting which has rarely been studied.
We finally illustrate the proficiency of our procedure on various synthetic and real datasets. Importantly, our algorithm is efficient for enforcing a pre-specified level of fairness.  
A natural way for further research is to extend our methodology to other notions of fairness such as {\it equalized odds} and also to consider settings of multi-category sensitive attributes. We believe that the present work is a relevant step to handle these two problems.
On the other hand, the calibration of the level of unfairness $\varepsilon\geq 0$ is an important empirical issue. As mentioned in the introduction, there are some heuristics that provide guidelines for its calibration but one may ask for more advanced and robust approaches. In particular, a future direction of research is to describe a methodology that statistically justifies a data-driven calibration of this parameter in order to optimally compromise risk and unfairness.

\bibliography{biblio}

\begin{thebibliography}{39}
\providecommand{\natexlab}[1]{#1}
\providecommand{\url}[1]{\texttt{#1}}
\expandafter\ifx\csname urlstyle\endcsname\relax
  \providecommand{\doi}[1]{doi: #1}\else
  \providecommand{\doi}{doi: \begingroup \urlstyle{rm}\Url}\fi

\bibitem[Adebayo and Kagal(2016)]{adebayo2016iterative}
J.~Adebayo and L.~Kagal.
\newblock Iterative orthogonal feature projection for diagnosing bias in
  black-box models.
\newblock In \emph{Conference on Fairness, Accountability, and Transparency in
  Machine Learning}, 2016.

\bibitem[Agarwal et~al.(2018)Agarwal, Beygelzimer, Dud{\'\i}k, Langford, and
  Wallach]{Agarwal_Beygelzimer_Dubik_Langford_Wallach18}
A.~Agarwal, A.~Beygelzimer, M.~Dud{\'\i}k, J.~Langford, and H.~Wallach.
\newblock A reductions approach to fair classification.
\newblock In \emph{Proceedings of the 35th International Conference on Machine
  Learning}, 2018.

\bibitem[Agarwal et~al.(2019)Agarwal, Dudik, and Wu]{agarwal2019fair}
A.~Agarwal, M.~Dudik, and Z.~S. Wu.
\newblock Fair regression: Quantitative definitions and reduction-based
  algorithms.
\newblock In \emph{International Conference on Machine Learning}, 2019.

\bibitem[Alghamdi et~al.(2022)Alghamdi, Hsu, Jeong, Wang, Michalak, Asoodeh,
  and Calmon]{alghamdi2022beyond}
W.~Alghamdi, H.~Hsu, H.~Jeong, H.~Wang, P.W. Michalak, S.~Asoodeh, and
  F.~Calmon.
\newblock Beyond adult and {COMPAS}: Fair multi-class prediction via
  information projection.
\newblock In \emph{In Neural Information Processing Systems}, 2022.

\bibitem[Audibert and Tsybakov(2007)]{Audibert_Tsybakov07}
J.~Y. Audibert and A.~Tsybakov.
\newblock Fast learning rates for plug-in classifiers.
\newblock \emph{The Annals of Statistics}, 35\penalty0 (2):\penalty0 608--633,
  2007.

\bibitem[Barocas and Selbst(2014)]{barocas2014datas}
S.~Barocas and A.~Selbst.
\newblock {Big Data's Disparate Impact}.
\newblock \emph{SSRN eLibrary}, 2014.

\bibitem[Barocas et~al.(2018)Barocas, Hardt, and
  Narayanan]{barocas-hardt-narayanan}
S.~Barocas, M.~Hardt, and A.~Narayanan.
\newblock \emph{Fairness and Machine Learning}.
\newblock fairmlbook.org, 2018.

\bibitem[Calders et~al.(2009)Calders, Kamiran, and
  Pechenizkiy]{calders2009building}
T.~Calders, F.~Kamiran, and M.~Pechenizkiy.
\newblock Building classifiers with independency constraints.
\newblock In \emph{IEEE international conference on Data mining}, 2009.

\bibitem[Calders et~al.(2013)Calders, Karim, Kamiran, Ali, and
  Zhang]{calders2013controlling}
T.~Calders, A.~Karim, F.~Kamiran, W.~Ali, and X.~Zhang.
\newblock Controlling attribute effect in linear regression.
\newblock In \emph{IEEE International Conference on Data Mining}, 2013.

\bibitem[Calmon et~al.(2017)Calmon, Wei, Vinzamuri, Ramamurthy, and
  Varshney]{calmon2017optimized}
F.~Calmon, D.~Wei, B.~Vinzamuri, K.~N. Ramamurthy, and K.~R. Varshney.
\newblock Optimized pre-processing for discrimination prevention.
\newblock In \emph{Neural Information Processing Systems}, 2017.

\bibitem[Chiappa et~al.(2020)Chiappa, Jiang, Stepleton, Pacchiano, Jiang, and
  Aslanides]{chiappa2020general}
S.~Chiappa, R.~Jiang, T.~Stepleton, A.~Pacchiano, H.~Jiang, and J.~Aslanides.
\newblock A general approach to fairness with optimal transport.
\newblock In \emph{AAAI}, 2020.

\bibitem[Chzhen et~al.(2019)Chzhen, Denis, Hebiri, Oneto, and
  Pontil]{Chzhen_Denis_Hebiri_Oneto_Pontil19}
E.~Chzhen, C.~Denis, M.~Hebiri, L.~Oneto, and M.~Pontil.
\newblock Leveraging labeled and unlabeled data for consistent fair binary
  classification.
\newblock In \emph{Advances in Neural Information Processing Systems}, 2019.

\bibitem[Chzhen et~al.(2020{\natexlab{a}})Chzhen, Denis, Hebiri, Oneto, and
  Pontil]{Chzhen_Denis_Hebiri_Oneto_Pontil20Recali}
E.~Chzhen, C.~Denis, M.~Hebiri, L.~Oneto, and M.~Pontil.
\newblock Fair regression via plug-in estimator and recalibrationwith
  statistical guarantees.
\newblock In \emph{Advances in Neural Information Processing Systems},
  2020{\natexlab{a}}.

\bibitem[Chzhen et~al.(2020{\natexlab{b}})Chzhen, Denis, Hebiri, Oneto, and
  Pontil]{Chzhen_Denis_Hebiri_Oneto_Pontil20TV}
E.~Chzhen, C.~Denis, M.~Hebiri, L.~Oneto, and M.~Pontil.
\newblock {Fair Regression via Plug-In Estimator and Recalibration}.
\newblock NeurIPS20, 2020{\natexlab{b}}.

\bibitem[Collins(2007)]{Collins07}
B.~Collins.
\newblock Tackling unconscious bias in hiring practices: The plight of the
  rooney rule.
\newblock \emph{NYU Law Review}, 82:\penalty0 870--912, 2007.

\bibitem[Dastin(2018)]{dastin2018amazon}
J.~Dastin.
\newblock Amazon scraps secret ai recruiting tool that showed bias against
  women.
\newblock In \emph{Ethics of Data and Analytics}, pages 296--299. Auerbach
  Publications, 2018.

\bibitem[Donini et~al.(2018{\natexlab{a}})Donini, Oneto, Ben-David,
  Shawe-Taylor, and Pontil]{Donini_Oneto_Ben-David_Taylor_Pontil18}
M.~Donini, L.~Oneto, S.~Ben-David, J.~S. Shawe-Taylor, and M.~Pontil.
\newblock Empirical risk minimization under fairness constraints.
\newblock In \emph{Neural Information Processing Systems}, 2018{\natexlab{a}}.

\bibitem[Donini et~al.(2018{\natexlab{b}})Donini, Oneto, Ben-David,
  Shawe-Taylor, and Pontil]{donini2018empirical}
M.~Donini, L.~Oneto, S.~Ben-David, J.~S. Shawe-Taylor, and M.~Pontil.
\newblock Empirical risk minimization under fairness constraints.
\newblock In \emph{Neural Information Processing Systems}, 2018{\natexlab{b}}.

\bibitem[Fehrman et~al.(2017)Fehrman, Muhammad, Mirkes, Egan, and
  Gorban]{fehrman2017five}
E.~Fehrman, A.~Muhammad, E.~Mirkes, V.~Egan, and A.~Gorban.
\newblock The five factor model of personality and evaluation of drug
  consumption risk.
\newblock In \emph{Data science}, pages 231--242. Springer, 2017.

\bibitem[Feldman et~al.(2015)Feldman, Friedler, Moeller, Scheidegger, and
  Venkatasubramanian]{feldman_certifying_2015}
M.~Feldman, S.~Friedler, J.~Moeller, C.~Scheidegger, and S.~Venkatasubramanian.
\newblock Certifying and removing disparate impact.
\newblock In \emph{Proceedings of the 21th {ACM} {SIGKDD} {International}
  {Conference} on {Knowledge} {Discovery} and {Data} {Mining}}, pages 259--268.
  ACM, 2015.

\bibitem[Fu et~al.(2019)Fu, Liu, and Guo]{Fu2019SequentialQP}
Z.~Fu, G.~Liu, and L.~Guo.
\newblock Sequential quadratic programming method for nonlinear least squares
  estimation and its application.
\newblock \emph{Mathematical Problems in Engineering}, 2019.

\bibitem[Gordaliza et~al.(2019)Gordaliza, Del~Barrio, Fabrice, and
  Loubes]{gordaliza2019obtaining}
P.~Gordaliza, E.~Del~Barrio, G.~Fabrice, and J.~M. Loubes.
\newblock Obtaining fairness using optimal transport theory.
\newblock In \emph{International Conference on Machine Learning}, 2019.

\bibitem[Hajian et~al.(2011)Hajian, Domingo-Ferrer, and
  Martínez-Ballesté]{Hajian_2011_discrimation}
S.~Hajian, J.~Domingo-Ferrer, and A.~Martínez-Ballesté.
\newblock Discrimination prevention in data mining for intrusion and crime
  detection.
\newblock In \emph{2011 IEEE Symposium on Computational Intelligence in Cyber
  Security (CICS)}, pages 47--54, 2011.

\bibitem[Hardt et~al.(2016{\natexlab{a}})Hardt, Price, and
  Srebro]{Hardt_Price_Srebro16}
M.~Hardt, E.~Price, and N.~Srebro.
\newblock Equality of opportunity in supervised learning.
\newblock In \emph{Neural Information Processing Systems}, 2016{\natexlab{a}}.

\bibitem[Hardt et~al.(2016{\natexlab{b}})Hardt, Price, and
  Srebro]{hardt2016equality}
M.~Hardt, E.~Price, and N.~Srebro.
\newblock Equality of opportunity in supervised learning.
\newblock In \emph{Neural Information Processing Systems}, 2016{\natexlab{b}}.

\bibitem[Holzer and Holzer(2000)]{HolzerHolzer00}
H.~Holzer and D.~Holzer.
\newblock Assessing affirmative action.
\newblock \emph{Journal of Economic Literature}, 38\penalty0 (3):\penalty0
  483--568, 2000.

\bibitem[Jiang et~al.(2019)Jiang, Pacchiano, Stepleton, Jiang, and
  Chiappa]{jiang2019wasserstein}
R.~Jiang, A.~Pacchiano, T.~Stepleton, H.~Jiang, and S.~Chiappa.
\newblock Wasserstein fair classification.
\newblock \emph{arXiv preprint arXiv:1907.12059}, 2019.

\bibitem[Kamiran et~al.(2013)Kamiran, Zliobaite, and
  Calders]{KamiranZC13RemoveIllegal}
F.~Kamiran, I.~Zliobaite, and T.~Calders.
\newblock Quantifying explainable discrimination and removing illegal
  discrimination in automated decision making.
\newblock \emph{Knowl. Inf. Syst.}, 35\penalty0 (3):\penalty0 613--644, 2013.

\bibitem[{Le Gouic} et~al.(2020){Le Gouic}, Loubes, and
  Rigollet]{gouic2020price}
T.~{Le Gouic}, {J.-M.} Loubes, and P.~Rigollet.
\newblock Projection to fairness in statistical learning.
\newblock \emph{arXiv preprint arXiv:2005.11720}, 2020.

\bibitem[Lum and Johndrow(2016)]{lum2016statistical}
K.~Lum and J.~Johndrow.
\newblock A statistical framework for fair predictive algorithms.
\newblock \emph{arXiv preprint arXiv:1610.08077}, 2016.

\bibitem[Nie(2007)]{nie2007sequential}
P.-y. Nie.
\newblock Sequential penalty quadratic programming filter methods for nonlinear
  programming.
\newblock \emph{Nonlinear Analysis: Real World Applications}, 8\penalty0
  (1):\penalty0 118--129, 2007.

\bibitem[Oneto et~al.(2019)Oneto, Donini, and Pontil]{oneto2019general}
L.~Oneto, M.~Donini, and M.~Pontil.
\newblock General fair empirical risk minimization.
\newblock \emph{arXiv preprint arXiv:1901.10080}, 2019.

\bibitem[Rubinstein(1999)]{rubinstein1999cross}
R.~Rubinstein.
\newblock The cross-entropy method for combinatorial and continuous
  optimization.
\newblock \emph{Methodology and computing in applied probability}, 1\penalty0
  (2):\penalty0 127--190, 1999.

\bibitem[Tavker et~al.(2020)Tavker, Ramaswamy, and Narasimhan]{Tavkeretal2020}
S.K. Tavker, H.G. Ramaswamy, and H.~Narasimhan.
\newblock Consistent plug-in classifiers for complex objectives and
  constraints.
\newblock In \emph{Advances in Neural Information Processing Systems},
  volume~33, pages 20366--20377, 2020.

\bibitem[Ye and Xie(2020)]{Ye_Xie_2020_fairness}
Q.~Ye and W.~Xie.
\newblock Unbiased subdata selection for fair classification: A unified
  framework and scalable algorithms.
\newblock \emph{arXiv preprint arXiv:2012.12356}, 2020.

\bibitem[Zafar et~al.(2017)Zafar, Valera, Gomez~Rodriguez, and
  Gummadi]{zafar2017fairness}
M.~B. Zafar, I.~Valera, M.~Gomez~Rodriguez, and K.~P. Gummadi.
\newblock Fairness beyond disparate treatment \& disparate impact: Learning
  classification without disparate mistreatment.
\newblock In \emph{International Conference on World Wide Web}, 2017.

\bibitem[Zafar et~al.(2019)Zafar, Valera, Gomez-Rodriguez, and
  Gummadi]{zafar2019fairness}
M.~B. Zafar, I.~Valera, M.~Gomez-Rodriguez, and K.~P. Gummadi.
\newblock Fairness constraints: A flexible approach for fair classification.
\newblock \emph{Journal of Machine Learning Research}, 20\penalty0
  (75):\penalty0 1--42, 2019.

\bibitem[Zemel et~al.(2013)Zemel, Wu, Swersky, Pitassi, and
  Dwork]{zemel2013learning}
R.~Zemel, Y.~Wu, K.~Swersky, T.~Pitassi, and C.~Dwork.
\newblock Learning fair representations.
\newblock In \emph{International Conference on Machine Learning}, 2013.

\bibitem[Zhang et~al.(2018)Zhang, Lemoine, and Mitchell]{Adversarial19}
B.~Hu Zhang, B.~Lemoine, and M.~Mitchell.
\newblock Mitigating unwanted biases with adversarial learning.
\newblock \emph{CoRR}, abs/1801.07593, 2018.

\end{thebibliography}

\newpage

\begin{center}

\Large{\bf Appendix}
\end{center}
\vspace*{0.25cm}

\appendix

In this section, we gather the proofs of our results. Section~\ref{sec:techResult} is devoted to useful technical results. 
In Section~\ref{sec:proofofSec2} we give the proof of the results related to the optimal fair predictors while Section~\ref{sec:proofofSec3} is dedicated to the theoretical properties of our estimation procedure.
Finally, we provide additional numerical results in Section~\ref{sec:add_numres}.
In all the sequel, $C$ denotes a generic constant, whose value may vary from line to line.

\section{Technical results}
\label{sec:techResult}

\begin{lem}[Hoeffding]
\label{lm:Hoeffding}
Let $Z\sim \mathcal{B}(N,p)$, with $p \in (0,1)$. We then have for all $t >0$ and $N > \frac{t}{p}$  
\begin{equation*}
\mathbb{P}(Z \leq t) \leq \exp{\left( -2N(p-t/N)^2\right)} .
\end{equation*}
\end{lem}

\begin{lem}
\label{lm:InverseBinom}
Let $Z\sim \mathcal{B}(N,p)$. We have that
\begin{equation*}
\mathbb{E}\left[\dfrac{\one_{\{Z \geq 1\}}}{Z}\right] \leq \dfrac{2}{(N+1)p}
\end{equation*}
\end{lem}

\begin{prop}
\label{prop:subgradMin}
Let $f: \; \mathbb{R}^M \rightarrow \mathbb{R}$ be a convex continuous function, and $\mathcal{H} \subset \mathbb{R}^M$ a closed convex set.
We consider the minimizer ${\bf x}^*$ of the function $f$ over the set $\mathcal{H}$
\begin{equation*}
{\bf x}^* \in \arg\min_{{\bf x} \in \mathcal{H}} f({\bf x}).
\end{equation*}
Then, there exists a subgradient $\mathfrak{h}$ in the subdifferential $\partial{f}({\bf x}^*)$ of $f$ at the point ${\bf x}^*$ such that
\begin{equation*}
\mathfrak{h}^T({\bf y}-{\bf x}^*) \geq 0, \;\; \forall {\bf y} \in \mathcal{H}.
\end{equation*}
\end{prop}
From the above proposition, it is easy to show the following result.
\begin{coro}
\label{coro:subgradMinInt}
Let $f: \; \mathbb{R}^M \rightarrow \mathbb{R}$ be a convex continuous function. Let $\mathcal{H} = \mathbb{R}_{+}^M$. We consider the minimizer ${\bf x}^*$ of the function $f$ over the set $\mathcal{H}$. Let $\mathcal{M} := \{m \in [M], \; {\bf x}_m^* \neq 0\}$. Then there exists a subgradient 
 $\mathfrak{h}\in \partial{f}({\bf x}^*)$, such that for all $m \in [M]$ we have $\mathfrak{h}_m \geq 0$ and in particular,
\begin{equation*}
\forall  m \in \mathcal{M}, \;\; \mathfrak{h}_m = 0.
\end{equation*}
\end{coro}

\section{Proof of Section~\ref{sec:FMCApprox}}
\label{sec:proofofSec2}

We begin with an auxiliary lemma, which provides an alternative useful representation of $\mathcal{R}_{\lambda^{(1)},\lambda^{(2)}}(g)$.

\begin{lem}
\label{lem:eqDecompLambdaEps}
Let $\varepsilon \geq 0$,
the $\varepsilon$-fair-risk of a classifier $g$ with tuning parameters $\lambda^{(1)} = (\lambda^{(1)}_1,\ldots, \lambda^{(1)}_K)\in\mathbb{R}_{+}^K, \lambda^{(2)} = (\lambda^{(2)}_1, \ldots, \lambda^{(2)}_K) \in \mathbb{R}_{+}^K$ reads as:
 \begin{equation}
\label{eq:eqDecompLambdaEps}
\mathcal{R}_{\lambda^{(1)}, \lambda^{(2)}}(g) = \sum_{s \in \mathcal{S}}\mathbb{E}_{X|S=s}\left[\sum_{k=1}^K \left(\pi_s p_k(X,S)-s(\lambda^{(1)}_k-\lambda^{(2)}_k)\right)\one_{\{g(X,S) \neq k\}}\right] - \varepsilon \sum_{k=1}^K(\lambda^{(1)}_k+\lambda^{(2)}_k).    
\end{equation}
\end{lem}

\begin{proof}[Proof of Lemma~\ref{lem:eqDecompLambdaEps}]

Let $(\lambda^{(1)} , \lambda^{(2)}) \in \mathbb{R}_{+}^{2K}$ and recall the following definition of the $\varepsilon$-fair risk 
\begin{equation}
\label{eq:proofRiskDecompo}
\begin{aligned}
\mathcal{R}_{\lambda^{(1)}, \lambda^{(2)}}(g) =  \mathbb{P}\left(g(X,S) \neq Y\right)
-  \sum_{k=1}^K \sum_{s \in \mathcal{S}} s(\lambda^{(1)}_k-\lambda^{(2)}_k)\, \mathbb{E}_{X|S=s}\left[ \one_{\{g(X,s) \neq k\}}\right]
- \varepsilon \sum_{k=1}^K(\lambda^{(1)}_k+\lambda^{(2)}_k).
\end{aligned}
\end{equation}
The result in~\eqref{eq:eqDecompLambdaEps} directly follows from the following decomposition
\begin{eqnarray*}
\mathbb{P}\left(g(X,S) \neq Y\right) & = & \sum_{k = 1}^K \mathbb{E}\left[\one_{\{g(X,S) \neq k\}}\one_{\{Y=k\}} \right]\\
& = & \sum_{k=1}^K\sum_{s \in \mathcal{S}}\mathbb{E}\left[\one_{\{g(X,S) \neq k\}}\one_{\{S=s\}}p_k(X,S)\right]\\
& = & \sum_{k=1}^K \sum_{s \in \mathcal{S}}\mathbb{E}_{X|S=s}\left[\one_{\{g(X,s) \neq k\}}\pi_s p_k(X,s)\right]\enspace.
\end{eqnarray*} 
\end{proof}

\begin{proof}[Proof of Theorem~\ref{thm:equivalenceApprox}]
The proof is divided into two parts. First, we provide the proof for $\varepsilon > 0$. Then the second part is dedicated to the proof of the result when $\varepsilon = 0$ which corresponds to the case of exact fairness.

\paragraph*{Proof for approximate fairness}
From Lemma~\ref{lem:eqDecompLambdaEps}, we deduce that $g^*_{\lambda^{(1)}, \lambda^{(2)}}$ should be defined for all $(x,s)\in \cal{X}\times \cal{S}$ as
\begin{equation}
\label{eq:eqOracleLambda}
g^*_{\lambda^{(1)}, \lambda^{(2)}}(x,s) = \arg\max_{k \in [K]}  \left(\pi_s p_k(X,S)-s(\lambda^{(1)}_k-\lambda^{(2)}_k)\right)\enspace,
\end{equation}
since it minimizes the risk $\mathcal{R}_{\lambda^{(1)}, \lambda^{(2)}}$.  Now we should maximize $\mathcal{R}_{\lambda^{(1)}, \lambda^{(2)}}(g^*_{\lambda^{(1)}, \lambda^{(2)}}) $ in the dual variables. Notice that the $\varepsilon$-fair risk can be written as
\begin{equation*}
\mathcal{R}_{\lambda^{(1)}, \lambda^{(2)}}(g^*_{\lambda^{(1)}, \lambda^{(2)}}) = 1-\sum_{s \in \mathcal{S}} \mathbb{E}_{X|S=s}
\left[\max_{k \in [K]} \left(\pi_s p_k(X,S)-s(\lambda^{(1)}_k-\lambda^{(2)}_k)\right)\right] - \varepsilon \sum_{k=1}^K(\lambda^{(1)}_k+\lambda^{(2)}_k)\enspace.
\end{equation*}
Hence, a maximizer $(\lambda^{*(1)}, \lambda^{*(2)})$ in $\mathbb{R}_{+}^{2K}$ of $(\lambda^{(1)}, \lambda^{(2)})\mapsto \mathcal{R}_{\lambda^{(1)},\lambda^{(2)}}(g^*_{\lambda^{(1)}, \lambda^{(2)}})$ is solution of
\begin{equation*}
(\lambda^{*(1)},\lambda^{*(2)}) \in \arg\min_{(\lambda^{(1)},\lambda^{(2)}) \in \mathbb{R}_{+}^{2K}}   \underbrace{\sum_{s \in \mathcal{S}}
\mathbb{E}_{X|S=s} \left[\max_{k \in [K]} \left(\pi_s p_k(X,s)-s(\lambda^{(1)}_k-\lambda^{(2)}_k)\right)\right]+\varepsilon \sum_{k=1}^K(\lambda^{(1)}_k+\lambda^{(2)}_k)}_{H(\lambda^{(1)},\lambda^{(2)})} \enspace.  
\end{equation*}
The rest of the proof consists in showing that such a calibration of the tuning parameters $(\lambda^{(1)}, \lambda^{(2)})$ implies that 
$g^*_{\lambda^{*(1)}, \lambda^{*(2)}}$ is indeed $\varepsilon$-fair.
Observe that
$$
 H(\lambda^{(1)},\lambda^{(2)}) \geq \varepsilon \sum_{k=1}^K (\lambda^{(1)}_k +\lambda^{(2)}_k)\enspace,
$$
and then $\lim_{\Vert (\lambda^{(1)},\lambda^{(2)}) \Vert_2^2 \to \infty}  H(\lambda^{(1)},\lambda^{(2)}) = + \infty$. Moreover, the mapping $H$ is continuous and convex in $ (\lambda^{(1)},\lambda^{(2)})$. Therefore the minimum $(\lambda^{*(1)},\lambda^{*(2)})$ exists, and there exists some constant $C_{\lambda}>0$ such that for all $k\in [K]$ and $j\in \{1,2\}$ we have $|\lambda_k^{(j)}| \leq C_{\lambda}$.

Let us derive a subgradient $\mathfrak{h}^*=(\mathfrak{h}^{*(1)}, \mathfrak{h}^{*(2)})$ of $H$ at the optimum $(\lambda^{*(1)}, \lambda^{*(2)})$ with $\mathfrak{h}^{*(1)} = \left(\mathfrak{h}_{1}^{*(1)}, \ldots,  \mathfrak{h}_{K}^{*(1)}\right) $ and $\mathfrak{h}^{*(2)} = \left(\mathfrak{h}_{1}^{*(2)}, \ldots,  \mathfrak{h}_{K}^{*(2)}\right)$ being two vectors in $\mathbb{R}^K$.
In order to express $\mathfrak{h}^*$ let us build the subdifferential of the function ${f} \left(x,(\lambda^{(1)}, \lambda^{(2)}) \right) := \max_{k\in[K]} \left\{ h_k^s\left(x,(\lambda^{(1)}, \lambda^{(2)}) \right) \right\}$ at the point $(\lambda^{*(1)}, \lambda^{*(2)})$ with
\begin{equation*}
 h_k^s\left(x,(\lambda^{(1)}, \lambda^{(2)}) \right)
= \pi_sp_k(x,s)- s\left(\lambda_k^{(1)}- \lambda_k^{(2)}\right).
\end{equation*}
We have that
\begin{multline*}
\partial {f}
\left(x,(\lambda^{*(1)}, \lambda^{*(2)}) \right) \\ = \conv \left\{ \nabla h_k^s \left(x,(\lambda^{*(1)}, \lambda^{*(2)}) \right) \ : \ h_k^s \left(x,(\lambda^{*(1)}, \lambda^{*(2)})\right) = \max_{j\in[K]} \left\{ h_j^s\left(x,(\lambda^{*(1)}, \lambda^{*(2)}) \right)  \right\}\right\},
\end{multline*}
where $\nabla h_k^s \left(x,(\lambda^{(1)}, \lambda^{(2)}) \right) \in \mathbb{R}^{2K}$
is the gradient of the function $h_k^s$ \emph{w.r.t.} $(\lambda^{(1)}, \lambda^{(2)})$.
Therefore, we deduce that a subgradient $\mathfrak{h}^*$ of $H$ at $(\lambda^{*(1)}, \lambda^{*(2)})$ can be expressed for each $k \in [K]$, and $l \in \{1,2\}$ as
\begin{multline*}
\mathfrak{h}_{k}^{*(l)} =   
(2l-3) \sum_{s\in \mathcal{S}} \left\{s
\mathbb{P}_{X|S=s}\left(
\forall j \neq k \,\,  (\pi_s p_k(X,s)-s (\lambda^{*(1)}_k-\lambda^{*(2)}_k))  > (\pi_s p_j(X,s)-s (\lambda^{*(1)}_j-\lambda^{*(2)}_j))
\right) \right.\\ 
  + s \, u_k^s \,  \mathbb{P}_{X|S=s}\left(\forall j \neq k \,\,  (\pi_s p_k(X,s)-s (\lambda^{*(1)}_k-\lambda^{*(2)}_k))   \geq (\pi_s p_j(X,s)-s (\lambda^{*(1)}_j-\lambda^{*(2)}_j)),\,\, \right.
 \\
 \left.\left. \exists j\neq k \,\,  (\pi_s p_k(X,s)-s (\lambda^{*(1)}_k-\lambda^{*(2)}_k)) =  (\pi_s p_j(X,s)-s (\lambda^{*(1)}_j-\lambda^{*(2)}_j))  \right)  \right\}  + \varepsilon
 \enspace,
\end{multline*}
with $u_k^s  \in [0,1]$ for all $k\in [K]$ and all $s\in \cal{S}$. Thanks to Assumption~\ref{ass:continuity}, $p_k(X,s)- p_j(X,s)$ has no atom for all $s\in \mathcal{S}$ and then the second part of the r.h.s. of the above equation vanishes and we have 
\begin{multline*}
\mathfrak{h}_{k}^{*(l)}     =  \\
   (2l-3) \sum_{s\in \mathcal{S}}  s
\mathbb{P}_{X|S=s}\left(
\forall j \neq k \,\,  (\pi_s p_k(X,s)-s (\lambda^{*(1)}_k-\lambda^{*(2)}_k))  > 
(\pi_s p_j(X,s)-s (\lambda^{*(1)}_j-\lambda^{*(2)}_j))
\right)   + \varepsilon,
\end{multline*}
 which can be written as
\begin{equation*}
 \mathfrak{h}_{k}^{*(l)} =  (2l-3) \sum_{s\in \mathcal{S}}  s
\mathbb{P}_{X|S=s}\left(g^*_{\lambda^{*(1)}, \lambda^{*(2)}}(X,S) = k\right)   + \varepsilon.
\end{equation*}

Now, we apply Corollary~\ref{coro:subgradMinInt} and deduce, from the above equation, that if
\begin{itemize}
\item[$\bullet$] $\lambda^{*(1)}_k \neq 0$ and $ \lambda^{*(2)}_k \neq 0$,  we then necessary have $\mathfrak{h}_{k}^{*(l)} =0$ for $l\in \{1,2\}$ and then
\begin{eqnarray*}
\mathbb{P}_{X|S=1}\left(g^{*}_{\lambda^{*(1)},\lambda^{*(2)}}(X,S) = k \right) - \mathbb{P}_{X|S=-1}\left(g^*_{\lambda^{*(1)},\lambda^{*(2)}}(X,S) = k \right)  & = & \varepsilon 
\\
\mathbb{P}_{X|S=1}\left(g^{*}_{\lambda^{*(1)},\lambda^{*(2)}}(X,S) = k \right) - \mathbb{P}_{X|S=-1}\left(g^*_{\lambda^{*(1)},\lambda^{*(2)}}(X,S) = k \right) & = & - \ \varepsilon \enspace,
\end{eqnarray*}
which leads to a contradiction.
\item[$\bullet$] $\lambda^{*(1)}_k = 0$ and $ \lambda^{*(2)}_k = 0$, we get
\begin{eqnarray*}
\mathbb{P}_{X|S=1}\left(g^{*}_{\lambda^{*(1)},\lambda^{*(2)}}(X,S) = k \right) - \mathbb{P}_{X|S=-1}\left(g^*_{\lambda^{*(1)},\lambda^{*(2)}}(X,S) = k \right) & \leq & \varepsilon 
\\
\mathbb{P}_{X|S=1}\left(g^{*}_{\lambda^{*(1)},\lambda^{*(2)}}(X,S) = k \right) - \mathbb{P}_{X|S=-1}\left(g^*_{\lambda^{*(1)},\lambda^{*(2)}}(X,S) = k \right) & \geq & - \ \varepsilon \enspace,
\end{eqnarray*}
which gives
\begin{equation*}
\left|\mathbb{P}_{X|S=1}\left(g^{*}_{\lambda^{*(1)},\lambda^{*(2)}}(X,S) = k \right) - \mathbb{P}_{X|S=-1}\left(g^*_{\lambda^{*(1)}, \lambda^{*(2)}}(X,S) = k \right) \right| \leq \varepsilon. 
\end{equation*}
\item[$\bullet$] Finally, if $\lambda^{*(1)}_k\lambda^{*(2)}_k = 0 \;\; {\rm and} \;\; \lambda^{*(1)}_k +\lambda^{*(2)}_k > 0 $, we get
\begin{equation*}
\left|\mathbb{P}_{X|S=1}\left(g^{*}_{\lambda^{*(1)},\lambda^{*(2)}}(X,S) = k \right) - \mathbb{P}_{X|S=-1}\left(g^*_{\lambda^{*(1)}, \lambda^{*(2)}}(X,S) = k \right) \right| = \varepsilon. 
\end{equation*}
\end{itemize}
Hence, we have shown that for each $k \in [K]$,
\begin{equation*}
\left|\mathbb{P}_{X|S=1}\left(g^{*}_{\lambda^{*(1)},\lambda^{*(2)}}(X,S) = k \right) - \mathbb{P}_{X|S=-1}\left(g^*_{\lambda^{*(1)}, \lambda^{*(2)}}(X,S) = k \right) \right| \leq \varepsilon,
\end{equation*}
which means that $g_{\lambda^{*(1)}, \lambda^{*(2)}}^*$ is $\varepsilon$-fair: $\mathcal{U}(g_{\lambda^{*(1)}, \lambda^{*(2)}}^*) \leq \varepsilon$.

Furthermore, we also have that for each $k \in [K]$, 
the vector $(\lambda^{*(1)}, \lambda^{*(2)})$ meets the following constraint
$\lambda^{*(1)}_k\lambda^{*(2)}_k = 0 \;\; {\rm and} \;\; \lambda^{*(1)}_k +\lambda^{*(2)}_k \geq 0$. Since parameters $(\lambda^{*(1)}, \lambda^{*(2)})$ are bounded, we then deduce that for any classifier $g$ (see for instance~\eqref{eq:proofRiskDecompo})
\begin{equation*}
\mathcal{R}_{\lambda^{*(1)}, \lambda^{*(2)}}(g) \leq R(g) + C \left(\mathcal{U}(g) - \varepsilon\right),
\end{equation*}
therefore, for any $g \in \mathcal{G}_{\varepsilon-{\rm fair}}$
\begin{equation}
\label{eq:eqBoundedRisk}
\mathcal{R}_{\lambda^{*(1)}, \lambda^{*(2)}}(g) \leq R(g).
\end{equation}
Besides, considering the three above cases, we notice that 
\begin{equation*}
\left|\mathbb{P}_{X|S=1}\left(g^{*}_{\lambda^{*(1)},\lambda^{*(2)}}(X,S) = k \right) - \mathbb{P}_{X|S=-1}\left(g^*_{\lambda^{*(1)}, \lambda^{*(2)}}(X,S) = k \right) \right| <  \varepsilon \Rightarrow \lambda_k^{*(1)}= \lambda_k^{*(2)} = 0.
\end{equation*}
Since $g_{\lambda^{*(1)}, \lambda^{*(2)}}^* \in\mathcal{G}_{\varepsilon-{\rm fair}}$,
the above equation and Equation~\eqref{eq:eqBoundedRisk} imply that for any $g \in \mathcal{G}_{\varepsilon-{\rm fair}}$
\begin{equation*}
R(g_{\lambda^{*(1)}, \lambda^{*(2)}}^*) = \mathcal{R}_{\lambda^{*(1)}, \lambda^{*(2)}}(g_{\lambda^{*(1)}, \lambda^{*(2)}}^*) \leq \mathcal{R}_{\lambda^{*(1)}, \lambda^{*(2)}}(g) \leq R(g),
\end{equation*}
which concludes the proof.

\paragraph*{Proof for exact fairness}
First, we apply Lemma~\ref{lem:eqDecompLambdaEps} with $\varepsilon = 0$ and then have
\begin{equation*}
\mathcal{R}_{\lambda^{(1)}, \lambda^{(2)}}(g^*_{\lambda^{(1)}, \lambda^{(2)}}) = 1-\sum_{s \in \mathcal{S}} \mathbb{E}_{X|S=s}
\left[\max_{k \in [K]} \left(\pi_s p_k(X,S)-s(\lambda^{(1)}_k-\lambda^{(2)}_k)\right)\right],
\end{equation*}
with 
\begin{equation*}
g^*_{\lambda^{(1)}, \lambda^{(2)}}(x,s) = \arg\max_{k \in [K]}  \left(\pi_s p_k(X,S)-s(\lambda^{(1)}_k-\lambda^{(2)}_k)\right)\enspace.
\end{equation*}
Therefore, it is not difficult to see that using the reparametrization 
\begin{equation}
\label{eq:proofreparameter}
\beta_k = \lambda^{(1)}_k-\lambda_k^{(2)}, \;\; k = 1,\ldots,K,
\end{equation}
we can write
\begin{equation}
\label{eq:riskPrecis}
\mathcal{R}_{\lambda^{(1)}, \lambda^{(2)}}(g^*_{\lambda^{(1)}, \lambda^{(2)}}) =  \mathcal{R}_{\beta}(g^*_{\beta}) = 1 - \sum_{s  \in \mathcal{S}}
\mathbb{E}_{X|S=s}\left[\max_{k \in [K]} \left(\pi_s p_k(X,s)-s \beta_k\right)\right] \enspace.
\end{equation}
 Hence, a maximizer $\beta^*$ in $\mathbb{R}^K$ of $\beta \mapsto \mathcal{R}_{\beta}(g^*_{\beta})$ takes the form 
\begin{equation*}
\beta^* \in \arg\min_{\beta \in \mathbb{R}^K}   \sum_{s \in \mathcal{S}}
\mathbb{E}_{X|S=s} \left[\max_{k \in [K]} \left(\pi_s p_k(X,s)-s\beta_k\right)\right]\enspace.  
\end{equation*}
The above criterion is convex in $\beta$. Therefore, first order optimality conditions for the minimization over $\beta$ of the above criterion imply that, for each $k\in [K]$,
\begin{equation*}
\begin{aligned}
    0 & =  \sum_{s\in \mathcal{S}}  s
\mathbb{P}_{X|S=s}\left(\forall j \neq k \,\,  (\pi_s p_k(X,s)-s \beta_k^*)  > (\pi_s p_j(X,s)-s \beta_j^*)  \right) \\ 
&
  + s u_k^s \mathbb{P}_{X|S=s}\left(\forall j \neq k \,\,  (\pi_s p_k(X,s)-s \beta_k^*)   \geq (\pi_s p_j(X,s)-s \beta_j^*),\,\, 
  \exists j\neq k \,\,  (\pi_s p_k(X,s)-s \beta_k^*) =  (\pi_s p_j(X,s)-s \beta_j^*) \right)\enspace,
\end{aligned}
\end{equation*}
with $u_k^s \in [0,1]$ for all $k\in [K]$ and $s \in \mathcal{S}$. As in the case where $\varepsilon>0$, we use Assumption~\ref{ass:continuity} on the distribution of $p_k(X,s)- p_j(X,s)$ to show that the second part of the r.h.s. vanishes.
Therefore for all $k \in [K]$ 
\begin{equation*}
\mathbb{P}_{X|S=1}\left(g^{*}_{\beta^*}(X,S) \neq k \right) = \mathbb{P}_{X|S=-1}\left(g^*_{\beta^*}(X,S) \neq k \right) \enspace,
\end{equation*}
meaning that the classifier $g^*_{\beta^*}$ is fair. Furthermore, for any fair classifier $g\in\mathcal{G}_{\rm fair}$, we observe that 
\begin{equation*}
\mathcal{R}(g^*_{\beta^*}) = \mathcal{R}_{\beta^*}(g^*_{\beta^*})  \leq  \mathcal{R}_{\beta^*}(g) = \mathcal{R}(g),   
\end{equation*}
so that $g^*_{\beta^*}$ is also an optimal fair classifier. 

Conversely, consider any optimal fair classifier $g^*_{\rm fair}\in\mathcal{G_{\rm fair}}$. Combining the fairness of $g^*_{\rm fair}$ with the optimality of $\beta^*$ over the family $(R_{\beta}(g^*_\beta))_{\beta\in\mathbf{R}^K}$, we deduce 
\begin{equation*}
\mathcal{R}_{\beta^*}(g^*_{\rm fair}) = \mathcal{R}(g^*_{\rm fair})  \leq 
\mathcal{R}_{\beta^*}(g^*_{\beta^*}) \leq \mathcal{R}_{\beta^*}(g),   \mbox{for any $g\in\mathcal{G}$}\,.
\end{equation*}
Hence any optimal fair classifier is a minimizer of $\mathcal{R}_{\beta^*}$ over $\mathcal{G}$. 

\end{proof}

\section{Proof of Section~\ref{sec:datadriven}}
\label{sec:proofofSec3}

We first introduce some notation. We recall that $N_{\min} = \min(N_1, N_{-1})$ and denote by $\hat{P}_{X|S=s}$ the empirical measure with respect to $(X^s_1, \ldots,X^s_{N_s})$ for $s \in \mathcal{S}$. 
Furthermore, throughout this section, we consider the following convention $\frac{0}{0} = 0$. Hence, if $N_s = 0$, we then have 
$\hat{P}_{X|S=s}(A) = 0$ for any event $A$.

We start this section with two results. Lemma-\ref{lm:PigeonHole} directly follows from similar arguments as in the proof of Lemma~B.8 in~\cite{Chzhen_Denis_Hebiri_Oneto_Pontil20Recali}. Its proof is hence omitted.
\begin{lem}
\label{lm:PigeonHole}
Conditional on the data, we have that, for each $s \in \mathcal{S}$ and $k \in [K]$,
\begin{eqnarray*}
\hat{\mathbb{P}}_{X|S=s}
\left(\exists j \neq k, \hat{h}_{k}^s(X, \hat{\lambda}^{(1)}_k,\hat{\lambda}^{(2)}_k) =  \hat{h}_{j}^s(X, \hat{\lambda}^{(1)}_j,\hat{\lambda}^{(2)}_j)\right) & = & 
\frac{1}{N_s} \sum_{i=1}^{N_s}
\one_{ \left\{\exists j \neq k, \hat{h}_{k}^s(X_i^{s}, \hat{\lambda}^{(1)}_k,\hat{\lambda}^{(2)}_k) =  \hat{h}_{j}^s(X_i^{s}, \hat{\lambda}^{(1)}_j,\hat{\lambda}^{(2)}_j)\right\}}
 \\ 
 & \leq &\frac{K-1}{N_s} \;\; {\it a.s.} \enspace,
\end{eqnarray*}
where $\hat{h}_{k}^s:(x,\lambda^{(1)}, \lambda^{(2)}) \mapsto \hat{\pi}_s \bar{p}_k(x,s)-s\left(\lambda^{(1)}-\lambda^{(2)} \right)$.
\end{lem}

\begin{lem}
\label{lm:DKWConsequence}
Let us introduce for all $k\in [K]$ the random variable
$$
\hat{A}_k = \left| \sum_{s\in \mathcal{S}}s
\left(\mathbb{P}_{X|S=s}-\hat{\mathbb{P}}_{X|S=s}\right)\left(\forall j \neq k \,\,  \hat{h}_{k}^s(X, \hat{\lambda}^{(1)}_k,\hat{\lambda}^{(2)}_k) > \hat{h}_{j}^s(X, \hat{\lambda}^{(1)}_j, \hat{\lambda}^{(2)}_j)
\right) \right| \enspace.
$$
Then all $k\in [K]$
\begin{enumerate}
    \item there exists $C_1>0$, that depends on $K$ such that 
    \begin{equation*}
\mathbb{E} \left[\hat{A}_k \one_{\{N_{\min} \geq 1\}}\ | \ \mathcal{D}_n,S_1,\ldots,S_N\right]  \leq \dfrac{C_1\one_{\{N_{\min} \geq 1\}}}{\sqrt{N_{\min}}} \enspace;
\end{equation*}
\item for all $\delta >0 $, the event $\mathcal{A}_k (\delta) = \left\{ \hat{A}_k \leq  K \sqrt{ \frac{2  \log(\frac{4K}{\delta})} {N_{\min}}} \right\} \bigcap \left\{N_{\min} \geq 1\right\}$ holds with probability greater than $1-\delta$.
\end{enumerate}
\end{lem}
\begin{proof}
\begin{enumerate}
\item For this part, we work on the event $\{N_{\min} \geq 1\}$ conditionally on $\mathcal{D}_n$ and on $S_1,\ldots,S_N$. For $s \in \{-1,1\}$, and $k \in [K]$, we have
\begin{multline*}
\label{eq:eqDKW1}
\left|\left(\mathbb{P}_{X|S=s}- \hat{\mathbb{P}}_{X|S=s}\right)\left(\forall j \neq k,\,\, \hat{h}_{k}^s(X, \hat{\lambda}^{(1)}_k,\hat{\lambda}^{(2)}_k) > \hat{h}_{j}^s(X, \hat{\lambda}^{(1)}_j,\hat{\lambda}^{(2)}_j\right) \right|  = \\ \left|\left(\mathbb{P}_{X|S=s}-
\hat{\mathbb{P}}_{X|S=s}\right)\left(\forall j \neq k,\,\,
\bar{p}_k(X,s) - \bar{p}_j(X,s) > \frac{s\left((\hat{\lambda}^{(1)}_k-\hat{\lambda}^{(2)}_k) -(\hat{\lambda}^{(1)}_j- (\hat{\lambda}^{(2)}_j)\right)}{\hat{\pi}_s}\right) \right| \\
\leq \sum_{j = 1}^K \sup_{t \in \mathbb{R}} 
\left|\left(\mathbb{P}_{X|S=s}-
\hat{\mathbb{P}}_{X|S=s}\right)\left(
\bar{p}_k(X,s) - \bar{p}_j(X,s) > t\right) \right| \enspace.
\end{multline*}
Therefore, from the Dvoretzky-Kiefer-Wolfowitz Inequality, 
we deduce that,
for each $s \in \mathcal{S}$ and $k \in [K]$
\begin{equation*}
\mathbb{E} \left[\hat{A}_k \one_{\{N_{\min} \geq 1\}}\ | \ \mathcal{D}_n,S_1,\ldots,S_N\right]  \leq \dfrac{C_1\one_{\{N_{\min} \geq 1\}}}{\sqrt{N_{\min}}}.
\end{equation*}

\item  From the Dvoretzky-Kiefer-Wolfowitz Inequality, conditional on $\mathcal{D}_n$ and on $(S_1, \ldots,S_N)$,
we have on the event $\{N_{\min} \geq 1\}$, for each $u > 0$ and for all $j,k\in [K]$, $s\in \mathcal{S}$, and $t>0$
\begin{equation*}
\mathbb{P}\left(\sup_{t \in \mathbb{R}} 
\left|\left(\mathbb{P}_{X|S=s}-
\hat{\mathbb{P}}_{X|S=s}\right)\left(
\bar{p}_k(X,s) - \bar{p}_j(X,s) > t\right) \right| \geq u\right) \leq 
2 \exp(-2N_{s}u^2) \leq 2\exp(-2N_{{\rm min}}u^2).
\end{equation*}
Since
\begin{equation*}
\hat{A}_k \leq \sum_{s \in \mathcal{S}}\sum_{j=1}^K \sup_{t \in \mathbb{R}} 
\left|\left(\mathbb{P}_{X|S=s}-
\hat{\mathbb{P}}_{X|S=s}\right)\left(
\bar{p}_k(X,s) - \bar{p}_j(X,s) > t\right) \right|,
\end{equation*}
we deduce for each $u> 0$ and $k \in [K]$
\begin{eqnarray*}
\mathbb{P}\left(\hat{A}_k \geq u \right) & \leq &  \sum_{s \in \mathcal{S}}\sum_{j=1}^K \mathbb{P}\left(\sup_{t \in \mathbb{R}} 
\left|\left(\mathbb{P}_{X|S=s}-
\hat{\mathbb{P}}_{X|S=s}\right)\left(
\bar{p}_k(X,s) - \bar{p}_j(X,s) > t\right) \right| \geq \frac{u}{2K}\right)\\
& \leq & 4K \exp\left(-\frac{u^2 N_{{\rm min}}}{2K^2}\right).
\end{eqnarray*}
Hence, from the above inequality, we obtain that 
\begin{equation*}
\one_{\{N_ {\min} \geq 1\}} \mathbb{P}\left(\hat{A}_k \geq K \sqrt{ \frac{2  \log(\frac{4K}{\delta})} {N_{\min}}} \; \middle |  \; \mathcal{D}_n,(S_1,\ldots,S_N) \right) \leq \one_{\{N_ {\min} \geq 1\}} \delta \leq \delta, 
\end{equation*}
which yields the desired result.

\end{enumerate}
\end{proof}

Let us now consider the proofs of Theorem~\ref{thm:unfairness} and Theorem~\ref{thm:excessRisk}.

\begin{proof}[Proof of Theorem~\ref{thm:unfairness}]
As in the proof of Theorem~\ref{thm:equivalenceApprox}],
we consider separately the cases of approximate ($\varepsilon>0$) and exact ($\varepsilon=0$) fairness.

\paragraph*{Unfairness control in the case of approximate fairness}
We first consider the case where $\varepsilon > 0$. 
As in Lemma \ref{lm:PigeonHole}, we first introduce, for $s \in \mathcal{S}$ and $k \in [K]$,
\begin{equation*}
\hat{h}_{k}^s:\left(x, \lambda^{(1)}, \lambda^{(2)}\right) \mapsto \hat{\pi}_s \bar{p}_k(x,s)-s \left(\lambda^{(1)}-\lambda^{(2)} \right)   \enspace.
\end{equation*}
By construction, the estimator $\bar{p}_k(X,S)$ is randomized and then satisfies an analog version of Assumption~\ref{ass:continuity}. Therefore for all $s \in \mathcal{S}$ and $k \in [K]$
\begin{equation}
\label{eq:proofhatgenscore}
\mathbb{P}_{X|S=s}\left(\hat{g}_{\varepsilon}(X,S) = k\right) = 
\mathbb{P}_{X|S=s}\left(\forall j \neq k, \hat{h}_{k}^s(X, \hat{\lambda}^{(1)}_k,\hat{\lambda}^{(2)}_k) > \hat{h}_{j}^s(X, \hat{\lambda}^{(1)}_j, \hat{\lambda}^{(2)}_j)\right)\enspace.
\end{equation}
Now, we consider similar arguments as in Proof of Theorem~\ref{thm:equivalenceApprox}. First we observe that
\begin{equation}
\label{eq:eqUnfairness1}
\hat{H}\left(\lambda^{(1)}, \lambda^{(2)}\right)) \geq \varepsilon \sum_{k=1}^K (\lambda^{(1)}_k+\lambda^{(2)}_k)\enspace, 
\end{equation}
where $\hat{H}$ is the empirical version of $H$ and is defined as
$$ \hat{H}(\lambda^{(1)}, \lambda^{(2)})
=
\sum_{s \in \mathcal{S}}
\hat{\mathbb{E}}_{X|S=s} \left[\max_{k \in [K]} \left(\pi_s \bar{p}_k(X,s)-s(\lambda^{(1)}_k-\lambda^{(2)}_k)\right)\right]+\varepsilon \sum_{k=1}^K(\lambda^{(1)}_k+\lambda^{(2)}_k)\enspace,
$$
with $\hat{\mathbb{E}}_{X|S=s} $ being the empirical expectation over the points $X_i$ from the dataset $\mathcal{D}'_N$ such that the sensitive attribute $S_i = s$. From Equation~\eqref{eq:eqUnfairness1}, we deduce that the minimizer  $(\hat{\lambda}^{(1)}, \hat{\lambda}^{(2)})$ exists and is bounded by some $C_{\lambda}' > 0$ which depends neither on $N$ nor on $n$.
Furthermore, we have that a subgradient $\hat{\mathfrak{h}}$ of $\hat{H}$ 
can be expressed for each $k \in [K]$ and $l \in \{1,2\}$ as follows
\begin{multline}
\hat{\mathfrak{h}}_{k}^{(l)} =   
(2l-3) \sum_{s\in \mathcal{S}} \left\{s
\hat{\mathbb{P}}_{X|S=s}\left(\forall j \neq k \,\,  \hat{h}_{k}^s(X, \hat{\lambda}^{(1)}_k,\hat{\lambda}^{(2)}_k) > \hat{h}_{j}^s(X, \hat{\lambda}^{(1)}_j, \hat{\lambda}^{(2)}_j)
\right) \right.\\ 
  + s \, u_k^s \,  \hat{\mathbb{P}}_{X|S=s}\left(\forall j \neq k \,\,  \hat{h}_{k}^s(X, \hat{\lambda}^{(1)}_k,\hat{\lambda}^{(2)}_k) \geq \hat{h}_{j}^s(X, \hat{\lambda}^{(1)}_j, \hat{\lambda}^{(2)}_j),\,\, \right.
 \\
 \left.\left. \exists j\neq k \,\,  \hat{h}_{k}^s(X, \hat{\lambda}^{(1)}_k,\hat{\lambda}^{(2)}_k) = \hat{h}_{j}^s(X, \hat{\lambda}^{(1)}_j, \hat{\lambda}^{(2)}_j)  \right)  \right\}  + \varepsilon
 \enspace,
 \label{eq:proofSubdiffenempiriq}
\end{multline}
with $u_k^s \in [0,1]$.
Applying Lemma~\ref{lm:PigeonHole}, we observe that the second term in r.h.s is such that
\begin{multline}
0 \leq  \hat{\mathbb{P}}_{X|S=s}\left(\forall j \neq k \,\,  \hat{h}_{k}^s(X, \hat{\lambda}^{(1)}_k,\hat{\lambda}^{(2)}_k) \geq \hat{h}_{j}^s(X, \hat{\lambda}^{(1)}_j, \hat{\lambda}^{(2)}_j),\,\, \right. 
 \\
 \left. \exists j\neq k \,\,  \hat{h}_{k}^s(X, \hat{\lambda}^{(1)}_k,\hat{\lambda}^{(2)}_k) = \hat{h}_{j}^s(X, \hat{\lambda}^{(1)}_j, \hat{\lambda}^{(2)}_j)  \right)   \leq \dfrac{K-1}{N_{\rm min}} \enspace. \label{eq:proofPigeoholl}
\end{multline}
Hereafter, we follow the same reasoning as in the proof of Theorem~\ref{thm:equivalenceApprox}. We use Corallary~\ref{coro:subgradMinInt} 
and consider the following cases for $k \in [K]$.
\begin{itemize}
\item[$\bullet$] if  $\hat{\lambda}^{(1)}_k = 0$, and $\hat{\lambda}^{(2)}_k = 0$, we deduce that
\begin{equation}
\label{eq:eqCondition2}
\left|\sum_s s \hat{\mathbb{P}}_{X|S=s}\left(\forall j \neq k \,\,  \hat{h}_{k}^s(X, \hat{\lambda}^{(1)}_k,\hat{\lambda}^{(2)}_k) > \hat{h}_{j}^s(X, \hat{\lambda}^{(1)}_j, \hat{\lambda}^{(2)}_j)
\right) \right| \leq \varepsilon+\dfrac{2(K-1)}{N_{\rm min}}.
\end{equation}
\item if there exists $l \in \{1,2\}$ such that $\hat{\lambda}_k^{(l)} \neq 0$, then
$\hat{\mathfrak{h}}_k^l  = 0$.
\end{itemize}
Let us now deal with the unfairness of $\hat{g}_{\varepsilon}$, recalled in~\eqref{eq:proofhatgenscore}. Bounding this quantity is a direct implication of the above lines. On the one hand, let $k \in [K]$ such that $\hat{\lambda}^{(1)}_k = 0$, and $\hat{\lambda}^{(2)}_k = 0$, then from Equation~\eqref{eq:eqCondition2}, we have
\begin{multline}
\label{eq:eqlambda0}
\left|\sum_{s\in \mathcal{S}}s
\mathbb{P}_{X|S=s}\left(\hat{g}_{\varepsilon}(X,S) = k
\right) \right| = 
\left|\sum_{s\in \mathcal{S}}s
\mathbb{P}_{X|S=s}\left(\forall j \neq k \,\,  \hat{h}_{k}^s(X, \hat{\lambda}^{(1)}_k,\hat{\lambda}^{(2)}_k) > \hat{h}_{j}^s(X, \hat{\lambda}^{(1)}_j, \hat{\lambda}^{(2)}_j)
\right) \right| \\
\leq \left| \sum_{s\in \mathcal{S}}s
\left(\mathbb{P}_{X|S=s}-\hat{\mathbb{P}}_{X|S=s}\right)\left(\forall j \neq k \,\,  \hat{h}_{k}^s(X, \hat{\lambda}^{(1)}_k,\hat{\lambda}^{(2)}_k) > \hat{h}_{j}^s(X, \hat{\lambda}^{(1)}_j, \hat{\lambda}^{(2)}_j)
\right)\right| 
+ \varepsilon + \dfrac{2(K-1)}{N_{\rm min}}.
\end{multline}
On the other hand, if for $k \in [K]$ there exists $l \in \{1,2\}$ such that $\hat{\mathfrak{h}}_k^l = 0$ then in view of Equation~\eqref{eq:proofSubdiffenempiriq}, we also deduce that
\begin{multline*}
\left|\sum_{s\in \mathcal{S}}s
\mathbb{P}_{X|S=s}\left(\hat{g}_{\varepsilon}(X,S) = k
\right) \right| \\
\leq \left| \sum_{s\in \mathcal{S}}s
\left(\mathbb{P}_{X|S=s}-\hat{\mathbb{P}}_{X|S=s}\right)\left(\forall j \neq k \,\,  \hat{h}_{k}^s(X, \hat{\lambda}^{(1)}_k,\hat{\lambda}^{(2)}_k) > \hat{h}_{j}^s(X, \hat{\lambda}^{(1)}_j, \hat{\lambda}^{(2)}_j)
\right)\right| 
+ \varepsilon + \dfrac{2(K-1)}{N_{\rm min}}.
\end{multline*}
Therefore, from the above inequalities,
taking the maximum over $k\in [K]$, we deduce from Lemma~\ref{lm:DKWConsequence} (point 1.) that 
conditional on $ \mathcal{D}_n$ and on $(S_1, \ldots,S_N)$, 
\begin{equation*}
\mathbb{E}\left[\mathcal{U}(\hat{g}_{\varepsilon})\right] 
\leq \varepsilon +
\left(\dfrac{K C_1}{\sqrt{{N_{\min}}}} + \dfrac{2(K-1)}{N_{\min}}\right) \one_{\{N_{\rm min} \geq 1\}} + \mathbb{E}\left[\mathcal{U}(\hat{g}_{\varepsilon})\one_{\{N_{\min} =0\}}\right]
\leq \varepsilon + \dfrac{c_1\one_{\{N_{\rm min} \geq 1\}}}{\sqrt{{N_{\min}}}}
+ C_K \mathbb{P}\left(N_{\min} = 0\right),
\end{equation*}
for some non negative constants $c_1$ and $C_K$ that depend on $K$. 
Now, we observe that 
\begin{equation*}
\mathbb{P}\left(N_{\min}= 0\right) = \mathbb{P}\left(N_1 =0\right)+\mathbb{P}\left(N_{-1}=0\right))\leq \exp(\log(1-\pi_1)N)+\exp(\log(1-\pi_{-1})N).
\end{equation*}
Therefore, applying Lemma~\ref{lm:InverseBinom}, we deduce that
\begin{equation*}
\mathbb{E}\left[\mathcal{U}(\hat{g}_{\varepsilon})\right] \leq \dfrac{C}{\sqrt{N\min(\pi_{-1},\pi_1)}}.
\end{equation*}
\paragraph*{Unfairness control in the case of exact fairness}
Along this proof, we need to adjust the notation as in the case of the optimal rule, \emph{c.f.}~\eqref{eq:proofreparameter}.
As in Lemma \ref{lm:PigeonHole}, we first introduce, for $s \in \mathcal{S}$ and $k \in [K]$,
\begin{equation*}
\hat{h}_{k}^s:(x, \beta) \mapsto \hat{\pi}_s \bar{p}_k(x,s)-s{\beta}\enspace.
\end{equation*}
By construction, the estimator $\bar{p}_k(X,S)$ satisfies Assumption~\ref{ass:continuity}, therefore for all $s \in \mathcal{S}$ and $k \in [K]$
\begin{equation*}
\mathbb{P}_{X|S=s}\left(\hat{g}(X,S) = k\right) = 
\mathbb{P}_{X|S=s}\left(\forall j \neq k, \hat{h}_{k}^s(X, \hat{\beta}_k) > \hat{h}_{j}^s(X, \hat{\beta}_j)\right)\enspace.
\end{equation*}
Considering the first order optimality conditions for $\hat\beta$, we can show that, for all $k \in [K]$ and $s \in \mathcal{S}$, there  exists $\alpha_k^s \in [-1,1]$ such that
\begin{multline*}
s\hat{\mathbb{P}}_{X|S=s}\left(\forall j \neq k, \hat{h}_{k}^s(X, \hat{\beta}_k) > \hat{h}_{j}^s(X, \hat{\beta}_j)\right) +\\ \alpha_k^s    \hat{\mathbb{P}}_{X|S=s}\left(\forall j \neq k, \hat{h}_{k}^s(X, \hat{\beta}_k) \geq  \hat{h}_{j}^s(X, \hat{\beta}_j), \; \exists j \neq k, \hat{h}_{k}^s(X, \hat{\beta}_k) =  \hat{h}_{j}^s(X, \hat{\beta}_j)\right)
= 0\enspace.
\end{multline*}
From the above equation, we deduce that
\begin{multline*}
\mathcal{U}(\hat{g})=
\max_{k= 1\ldots,K}\left|  
\mathbb{P}_{X|S=1}\left(\hat{g}(X,S) = k\right) - \mathbb{P}_{X|S=-1}\left(\hat{g}(X,S) = k\right) \right|
\\ 
\leq \max_{k=1,\ldots,K}\sum_{s \in \mathcal{S}} \left|\left(\mathbb{P}_{X|S=s}- \hat{\mathbb{P}}_{X|S=s}\right)\left(\forall j \neq k, \,\, \hat{h}_{k}^s(X, \hat{\beta}_k) > \hat{h}_{j}^s(X, \hat{\beta}_j)\right) \right| \\+
\max_{k=1,\ldots,K}\sum_{s \in \mathcal{S}} \hat{\mathbb{P}}_{X|S=s}\left( \exists j \neq k,\,\, \hat{h}_{k}^s(X, \hat{\beta}_k) =  \hat{h}_{j}^s(X, \hat{\beta}_j)\right)\enspace.
\end{multline*}
Observe that for all $k\in[K]$
\begin{multline*}
\left|\left(\mathbb{P}_{X|S=s}- \hat{\mathbb{P}}_{X|S=s}\right)\left(\forall j \neq k,\,\, \hat{h}_{k}^s(X, \hat{\beta}_k) > \hat{h}_{j}^s(X, \hat{\beta}_j)\right) \right|  = \\ \left|\left(\mathbb{P}_{X|S=s}-
\hat{\mathbb{P}}_{X|S=s}\right)\left(\forall j \neq k,\,\,
\bar{p}_k(X,s) - \bar{p}_j(X,s) \geq \frac{s(\hat{\beta}_k-\hat{\beta}_j)}{\hat{\pi}_s}\right) \right| \\
\leq \sum_{j = 1}^K \sup_{t \in \mathbb{R}} 
\left|\left(\mathbb{P}_{X|S=s}-
\hat{\mathbb{P}}_{X|S=s}\right)\left(
\bar{p}_k(X,s) - \bar{p}_j(X,s) \geq t\right) \right| \enspace.
\end{multline*}
Therefore, from the Dvoretzky-Kiefer-Wolfowitz Inequality conditional on $\mathcal{D}_n$ and on $(S_1, \ldots,S_N)$, we deduce that,
for each $s \in \mathcal{S}$ and $k \in [K]$
\begin{equation*}
\mathbb{E}\left[\left|\left(\mathbb{P}_{X|S=s}- \hat{\mathbb{P}}_{X|S=s}\right)\left(\forall j \neq k, \hat{h}_{k}^s(X, \hat{\beta}_k) > \hat{h}_{j}^s(X, \hat{\beta}_j)\right) \right|\right] \leq 
C \sqrt{\dfrac{1}{N_s}} \enspace.
\end{equation*}
Applying Lemma~\ref{lm:PigeonHole}, we then get that, 
conditional on $\mathcal{D}_n$ and on $(S_1, \ldots,S_N)$, we have that
\begin{equation*}
\mathbb{E}\left[\mathcal{U}(\hat{g})\right] \leq  C  \sum_{s \in \mathcal{S}}   \sqrt{\dfrac{1}{N_s}} 
\enspace, 
\end{equation*}
for some positive constant $C$ that depends in $K$.
Since $N_s$ is a binomial random variable with parameters $N$ and $\pi_s$, we get
\begin{equation*}
\mathbb{E}\left[\mathcal{U}(\hat{g})\right] \leq
C \sqrt{\dfrac{1}{N}},
\end{equation*}
where $C$ depends on $K$ and $\min(\pi_{-1}, \pi_1)$.
\end{proof}

\begin{proof}[Proof of Theorem~\ref{thm:caracEpsfairEstimator}]
Let $0< \delta < 1$ and let $k \in [K]$. From Equations~\eqref{eq:proofSubdiffenempiriq}, and~\eqref{eq:proofPigeoholl} and using Corallary~\ref{coro:subgradMinInt}, we deduce that if 
$\lambda_k^{(1)} \neq 0, \lambda_k^{(2)} \neq 0$, then
\begin{eqnarray*}
\sum_{s\in \mathcal{S}}s
\hat{\mathbb{P}}_{X|S=s}\left(\forall j \neq k \,\,  \hat{h}_{k}^s(X, \hat{\lambda}^{(1)}_k,\hat{\lambda}^{(2)}_k) > \hat{h}_{j}^s(X, \hat{\lambda}^{(1)}_j, \hat{\lambda}^{(2)}_j)
\right) & \geq & \varepsilon - \dfrac{2(K-1)}{N_{\min}}\\
\sum_{s\in \mathcal{S}}s
\hat{\mathbb{P}}_{X|S=s}\left(\forall j \neq k \,\,  \hat{h}_{k}^s(X, \hat{\lambda}^{(1)}_k,\hat{\lambda}^{(2)}_k) > \hat{h}_{j}^s(X, \hat{\lambda}^{(1)}_j, \hat{\lambda}^{(2)}_j)
\right) & \leq & -\varepsilon + \dfrac{2(K-1)}{N_{\min}}\enspace. 
\end{eqnarray*}
Therefore, since $\dfrac{C_{\delta}}{\sqrt{N_{\min}}} \geq \dfrac{2(K-1)}{N_{\min}}$, we deduce that on $\mathcal{A}_{\min} = \left\{\varepsilon > \dfrac{C_{\delta}}{\sqrt{N_{\min}}} \right\}$ 
\begin{equation*}
0<\sum_{s\in \mathcal{S}}s
\hat{\mathbb{P}}_{X|S=s}\left(\forall j \neq k \,\,  \hat{h}_{k}^s(X, \hat{\lambda}^{(1)}_k,\hat{\lambda}^{(2)}_k) > \hat{h}_{j}^s(X, \hat{\lambda}^{(1)}_j, \hat{\lambda}^{(2)}_j)
\right) <0 \enspace,
\end{equation*}
which leads to a contradiction. Therefore, on the event $\mathcal{A}_{\min}$,
we necessary have $\hat{\lambda}^{(1)}_k \hat{\lambda}^{(2)}_k = 0 \;\; {\rm and} \;\; \hat{\lambda}^{(1)}_k + \hat{\lambda}^{(2)}_k \geq 0$.
Note that on the event $\mathcal{A}_{\min}$, we have $N_{\min} \geq 1$.

The remaining of the proof consists in dealing with the two sub-cases when $\hat{\lambda}^{(1)}_k \hat{\lambda}^{(2)}_k = 0 \;\; {\rm and} \;\; \hat{\lambda}^{(1)}_k + \hat{\lambda}^{(2)}_k \geq 0$.
First, let us consider
for $k \in [K]$, the case where $\hat{\lambda}^{(1)}_k \neq 0$, and $\hat{\lambda}^{(2)}_k = 0$ (the case $\hat{\lambda}^{(1)}_k = 0$, and $\hat{\lambda}^{(2)}_k \neq 0$ follows in the same way). We observe that since 
$\hat{\mathfrak{h}}_k^1 = 0$, on the event $\mathcal{A}_{\min}$
\begin{equation}
\label{eq:eqSubgradientCase1}
0   \leq \varepsilon-\dfrac{2(K-1)}{N_{\min}} \leq \sum_s s \hat{\mathbb{P}}_{X|S=s}\left(\forall j \neq k \,\,  \hat{h}_{k}^s(X, \hat{\lambda}^{(1)}_k,\hat{\lambda}^{(2)}_k) > \hat{h}_{j}^s(X, \hat{\lambda}^{(1)}_j, \hat{\lambda}^{(2)}_j)
\right) \leq \varepsilon + \dfrac{2(K-1)}{N_{\rm min}} \enspace.
\end{equation}
Moreover, we have that for each $k \in [K]$ such that
$\hat{\lambda}^{(1)}_k \neq 0$, and $\hat{\lambda}^{(2)}_k = 0$ 
\begin{multline*}
\left| \, \left| 
\sum_s s \mathbb{P}_{X|S=s}\left(\hat{g}_{\varepsilon}(X,S)=k\right)\right| - \varepsilon\right| 
= \\
\left| \, \left|\sum_s s \mathbb{P}_{X|S=s}\left(\forall j \neq k \,\,  \hat{h}_{k}^s(X, \hat{\lambda}^{(1)}_k,\hat{\lambda}^{(2)}_k) > \hat{h}_{j}^s(X, \hat{\lambda}^{(1)}_j, \hat{\lambda}^{(2)}_j)
\right)\right| \right. - \\ \qquad\qquad \left|\sum_s s \hat{\mathbb{P}}_{X|S=s}\left(\forall j \neq k \,\,  \hat{h}_{k}^s(X, \hat{\lambda}^{(1)}_k,\hat{\lambda}^{(2)}_k) > \hat{h}_{j}^s(X, \hat{\lambda}^{(1)}_j, \hat{\lambda}^{(2)}_j)
\right) \right|  +\\  \left. 
\left|\sum_s s \hat{\mathbb{P}}_{X|S=s}\left(\forall j \neq k \,\,  \hat{h}_{k}^s(X, \hat{\lambda}^{(1)}_k,\hat{\lambda}^{(2)}_k) > \hat{h}_{j}^s(X, \hat{\lambda}^{(1)}_j, \hat{\lambda}^{(2)}_j)
\right)\right| - \varepsilon\right|\enspace.
\end{multline*}
Using first the triangle inequality and then the reverse triangle inequality, we get from Equation~\eqref{eq:eqSubgradientCase1}
\begin{multline}
\label{eq:proofcase0neq0}
\left| \, \left| 
\sum_s s \mathbb{P}_{X|S=s}\left(\hat{g}_{\varepsilon}(X,S)=k\right)\right| - \varepsilon\right| \leq \\
\left| \sum_{s\in \mathcal{S}}s
\left(\mathbb{P}_{X|S=s}-\hat{\mathbb{P}}_{X|S=s}\right)\left(\forall j \neq k \,\,  \hat{h}_{k}^s(X, \hat{\lambda}^{(1)}_k,\hat{\lambda}^{(2)}_k) > \hat{h}_{j}^s(X, \hat{\lambda}^{(1)}_j, \hat{\lambda}^{(2)}_j)
\right) \right|+
\dfrac{2(K-1)}{N_{\rm min}} \enspace, 
\end{multline}
which yields together with Lemma~\ref{lm:DKWConsequence} (point 2.),
\begin{equation}
\label{eq:casNotEqual}
\left| \, \left| 
\sum_s s \mathbb{P}_{X|S=s}\left(\hat{g}_{\varepsilon}(X,S)=k\right)\right| - \varepsilon\right| \leq  \left( K \sqrt{ \frac{2  \log(\frac{4K}{\delta})} {N_{\min}}}
+
\dfrac{2(K-1)}{N_{\min}}\right)  \leq 
 \dfrac{ C_{\delta}}{\sqrt{N_{\min}}} \enspace.
\end{equation}

Now, observe that for the second sub-case when $\hat{\lambda}^{(1)}_k =  \hat{\lambda}^{(2)}_k = 0 $, we get using Equation~\eqref{eq:eqlambda0}, applying again Lemma~\ref{lm:DKWConsequence} (point 2.) the following bound
\begin{equation*}
\left| 
\sum_s s \mathbb{P}_{X|S=s}\left(\hat{g}_{\varepsilon}(X,S)=k\right)\right| \leq \varepsilon  + 
 \dfrac{ C_{\delta}}{\sqrt{N_{\min}}} \enspace.
\end{equation*}
Combining these two bounds, we conclude that  
on the event $\mathcal{A}(\delta) = \mathcal{A}_{\rm min} \bigcap \left(\cap_{k \in [K]} \mathcal{A}_k(\delta)\right)$
\begin{equation*}
\mathcal{U}(\hat{g}_{\varepsilon}) <  \varepsilon + \frac{C_{\delta}}{\sqrt{N_{\min}}} \enspace,
\end{equation*}
which concludes the main part of the proof. 
Let us now focus on the particular case where on  $\mathcal{A}(\delta)$ we have
\begin{equation*}
\mathcal{U}(\hat{g}_{\varepsilon}) <  \varepsilon - \frac{C_{\delta}}{\sqrt{N_{\rm min}}} \enspace.
\end{equation*}
Then for all $k\in [K] $ we have 
\begin{equation*}
\left| 
\sum_s s \mathbb{P}_{X|S=s}\left(\hat{g}_{\varepsilon}(X,S)=k\right)\right| < 
 \varepsilon - \frac{C_{\delta}}{\sqrt{N_{\rm min}}} \enspace.
\end{equation*}
Hence on the set $\mathcal{A}(\delta)$, we deduce that the case related to~\eqref{eq:casNotEqual} is not possible and then for each $k$, we necessary have $\hat{\lambda}^{(1)}_k =  \hat{\lambda}^{(2)}_k = 0$ and then 
\begin{equation*}
\hat{g}  = \hat{g}_{\varepsilon} \enspace.
\end{equation*}
To conclude the proof, we observe that
\begin{equation*}
\mathbb{P}\left(\mathcal{A}(\delta)^{c}\right) =\mathbb{P}\left(\mathcal{A}^{c}_{\min}\right) + \sum_{k=1}^K\mathbb{P}\left(\mathcal{A}_k^{c}(\delta)\right) \leq \mathbb{P}\left(N_1 \leq \frac{C_{\delta}^2}{\varepsilon^2}\right)  + \mathbb{P}\left(N_-1 \leq \frac{C_{\delta}^2}{\varepsilon^2}\right)+K\delta. 
\end{equation*}
But, from Lemma~\ref{lm:Hoeffding},
we have for each $s \in \mathcal{S}$,
\begin{equation*}
\mathbb{P}\left(N_s \leq \dfrac{C_{\delta}^2}{\varepsilon^2}\right) \leq \exp\left(-2N\left(\pi_{s}-\dfrac{C_{\delta}^2}{\varepsilon^2 N}\right)^2\right) \leq \exp\left(-N\frac{\pi^2_{s}}{2}\right) \leq \delta \enspace,
\end{equation*}
provided that $\pi_{s} > \dfrac{2C_{\delta}^2}{N\varepsilon^2} $, and $N \geq 2\frac{\log(1/\delta)}{\pi_{\min}^2}$. Since $\varepsilon > \frac{\sqrt{2}C_{\delta}}{\sqrt{\pi_{\min}N}}$, and $N \geq 2\frac{\log(1/\delta)}{\pi_{\min}^2}$, the latter conditions are satisfied. 
Therefore, we deduce that
\begin{equation*}
\mathbb{P}\left(\mathcal{A}(\delta)^{c}\right) \leq (K+2)\delta \enspace.
\end{equation*}
\end{proof}

\begin{proof}[Proof of Theorem~\ref{thm:excessRisk}]

We only consider the case $\varepsilon > 0$. The proof in the case of exact fairness relies on similar arguments and then it is omitted.
To ease the notation, we write $\hat{g}$ instead of $\hat{g}_{\varepsilon}$.

The proof goes conditional on the training data. First, let us decompose the \emph{excess fair-risk} of the classifier $\hat g$ in a convenient way for our analysis
\begin{eqnarray}
\label{eq:eqDecomp1eps}
\mathcal{R}_{\lambda^{*(1)}, \lambda^{*(2)}}(\hat{g})  - \mathcal{R}_{\lambda^{*(1)}, \lambda^{*(2)}}(g_{\varepsilon-\rm fair}^*)  & = & 
\left(\mathcal{R}_{\lambda^{*(1)}, \lambda^{*(2)}}\left(\hat{g}\right)
- \mathcal{R}_{\hat{\lambda}^{(1)}, \hat{\lambda}^{(2)}} (\hat{g})\right)  \nonumber 
\\ && + \left(\mathcal{R}_{\hat{\lambda}^{(1)}, \hat{\lambda}^{(2)}}(\hat{g}) - \mathcal{R}_{\lambda^{*(1)}, \lambda^{*(2)}}\left(g^*_{\lambda^{*(1)}, \lambda^{*(2)}}\right)\right)\enspace.
\end{eqnarray}
According to the first term, we have
\begin{multline*}
\left(\mathcal{R}_{\lambda^{*(1)}, \lambda^{*(2)}}\left(\hat{g}\right)
- \mathcal{R}_{\hat{\lambda}^{(1)}, \hat{\lambda}^{(2)}} (\hat{g})\right)
=  \sum_{k=1}^K\left(\lambda^{*(1)}_k-\hat{\lambda}_k^{(1)}\right)\left[\sum_{s\in \mathcal{S}}s\mathbb{P}_{X|S=s}\left(\hat{g}(X,S) =k\right)-\varepsilon\right]\\
+ \sum_{k=1}^K\left(\lambda^{*(2)}_k-\hat{\lambda}_k^{(2)}\right)\left[-\sum_{s\in \mathcal{S}}s\mathbb{P}_{X|S=s}\left(\hat{g}(X,S) =k\right)
-\varepsilon \right].
\end{multline*}
Let $\delta =1/N$. If $\varepsilon \leq \frac{\sqrt{2}C_{\delta}}{\sqrt{\pi_{\min}N}}$,
since parameters $\lambda^{*(l)}_k$ and $\hat{\lambda}_k^{(l)}$ are bounded, we deduce from the above equation and Theorem~\ref{thm:unfairness} that
\begin{equation}
\label{eq:eqrisk1}
\mathbb{E}\left[\left(\mathcal{R}_{\lambda^{*(1)}, \lambda^{*(2)}}\left(\hat{g}\right)
- \mathcal{R}_{\hat{\lambda}^{(1)} , \hat{\lambda}^{(2)}} (\hat{g})\right)\right] \leq C\min\left(\varepsilon, \frac{\sqrt{2}C_{\delta}}{\sqrt{\pi_{\min}N}}\right) + \dfrac{C}{\sqrt{N}} \leq 
C \dfrac{\log(N)}{\sqrt{N}} \enspace,
\end{equation}
where $C > 0$ is a constant which depends on $\pi_{\min}$ and $K$. 
If $\varepsilon > \frac{\sqrt{2}C_{\delta}}{\sqrt{\pi_{\min}N}}$, we apply Theorem~\ref{thm:caracEpsfairEstimator}. 
We have on the event $\mathcal{A}(1/N)$ that 
\begin{itemize}
\item either $\hat{\lambda}_k^{(1)}=0$, and then since $\lambda_k^{*(1)} > 0$ is bounded
\begin{equation*}
\left(\lambda^{*(1)}_k-\hat{\lambda}_k^{(1)}\right)\left[\sum_{s\in \mathcal{S}}s\mathbb{P}_{X|S=s}\left(\hat{g}(X,S) =k\right)-\varepsilon\right] \leq C \left(\mathcal{U}(\hat{g})-\varepsilon)\right)\leq C\dfrac{\log(N)}{\sqrt{N_{\min}}}\enspace.
\end{equation*}
\item or $\hat{\lambda}_k^{(1)} > 0$, in this case on $\mathcal{A}(1/N)$, $\sum_{s\in \mathcal{S}}s\mathbb{P}_{X|S=s}\left(\hat{g}(X,S) =k\right) > 0$. From Equation~\eqref{eq:casNotEqual} in the proof of Theorem~\ref{thm:caracEpsfairEstimator}, we deduce
\begin{equation*}
\left(\lambda^{*(1)}_k-\hat{\lambda}_k^{(1)}\right)\left[\sum_{s\in \mathcal{S}}s\mathbb{P}_{X|S=s}\left(\hat{g}(X,S) =k\right)-\varepsilon\right] \leq C\dfrac{\log(N)}{\sqrt{N_{\min}}}\enspace.
\end{equation*}
\end{itemize}
Since on $\mathcal{A}(1/N)$, $N_{\min} \geq 1$, we deduce that if $\varepsilon > \frac{\sqrt{2}C_{\delta}}{\sqrt{\pi_{\min}N}}$
\begin{equation*}
\mathbb{E}\left[\sum_{k=1}^K\left(\lambda^{*(1)}_k-\hat{\lambda}_k^{(1)}\right)\left[\sum_{s\in \mathcal{S}}s\mathbb{P}_{X|S=s}\left(\hat{g}(X,S) =k\right)-\varepsilon\right]\right] \leq 
C\left(\mathbb{E}\left[\dfrac{\log(N)\one{\{N_{\min \geq 1\}}}}{\sqrt{ N_{\min}}}\right] + \mathbb{P}\left(\mathcal{A}(1/N)^{c}\right)\right).
\end{equation*}
According to Lemma~\ref{lm:DKWConsequence} we have $\mathbb{P}\left(\mathcal{A}(1/N)^{c}\right) \leq \dfrac{K+2}{N}$. Then we deduce form Lemma~\ref{lm:InverseBinom} that
\begin{equation*}
\mathbb{E}\left[\sum_{k=1}^K\left(\lambda^{*(1)}_k-\hat{\lambda}_k^{(1)} \right)\left[\sum_{s\in \mathcal{S}}s\mathbb{P}_{X|S=s}\left(\hat{g}(X,S) =k\right)-\varepsilon\right]\right] \leq C\dfrac{\log(N)}{\sqrt{N}}.
\end{equation*}
Similar reasoning leads to
\begin{equation*}
\mathbb{E}\left[\sum_{k=1}^K\left(\lambda^{*(2)}_k-\hat{\lambda}_k^{(2)}\right)\left[-\sum_{s\in \mathcal{S}}s\mathbb{P}_{X|S=s}\left(\hat{g}(X,S) =k\right)-\varepsilon\right]\right] \leq C\dfrac{\log(N)}{\sqrt{N}}.
\end{equation*}
Combining the two above inequalities and Equation~\eqref{eq:eqrisk1},
we obtain for $\varepsilon > 0$ and $N$ large enough
\begin{equation}
\label{eq:eqrisk2}
\mathbb{E}\left[\left(\mathcal{R}_{\lambda^{*(1)}, \lambda^{*(2)}}\left(\hat{g}\right)
- \mathcal{R}_{\hat{\lambda}^{(1)}, \hat{\lambda}^{(2)}} (\hat{g})\right)\right] \leq
C \dfrac{\log(N)}{\sqrt{N}}.
\end{equation}

Then we have shown that the first term in the r.h.s. of Eq.~\eqref{eq:eqDecomp1eps} relies on the unfairness of the classifier $\hat{g}$. 
Now, let us consider the second term in r.h.s. of Equation~\eqref{eq:eqDecomp1eps}.
Our goal will be to show that this term mainly depends on the quality of the base estimators $\hat{p}_k$.
Since $\left(\lambda^{*(1)}, \lambda^{*(2)}\right)$ is a maximizer of $\mathcal{R}_{(\lambda^{(1)},\lambda^{(2)})} (g^*_{(\lambda^{(1)},\lambda^{(2)})})$ over $(\lambda^{(1)},\lambda^{(2)})$, it is clear that, conditional on the data,
$\mathcal{R}_{\lambda^{*(1)}, \lambda^{*(2)}} (g^*_{\lambda^{*(1)}, \lambda^{*(2)}}) \geq \mathcal{R}_{\hat{\lambda}^{(1)}, \hat{\lambda}^{(2)}}(g^*_{\hat{\lambda}^{(1)}, \hat{\lambda}^{(2)}})$. (The parameter $\hat{\lambda}$ is seen as fixed conditional on the data.)
Therefore, we have
\begin{equation*}
\mathcal{R}_{\hat{\lambda}^{(1)}, \hat{\lambda}^{(2)}}(\hat{g}) - \mathcal{R}_{\lambda^{*(1)}, \lambda^{*(2)}}\left(g^*_{\lambda^{*(1)}, \lambda^{*(2)}}\right) \leq \mathcal{R}_{\hat{\lambda}^{(1)}, \hat{\lambda}^{(2)}}(\hat{g}) - \mathcal{R}_{\hat{\lambda}^{(1)}, \hat{\lambda}^{(2)}}(g^*_{\hat{\lambda}^{(1)}, \hat{\lambda}^{(2)}})\enspace.
\end{equation*}
By introducing $\hat{g}^*_{\hat{\lambda}^{(1)}, \hat{\lambda}^{(2)}}$, we remove the estimation of $\lambda^{*(1)}, \lambda^{*(2)}$ from the study of $\mathcal{R}_{\hat{\lambda}^{(1)}, \hat{\lambda}^{(2)}}(\hat{g}) - \mathcal{R}_{\lambda^{*(1)}, \lambda^{*(2)}}\left(g^*_{\lambda^{*(1)}, \lambda^{*(2)}}\right)$. At this point, it becomes clear that bounding this term does not relies on the unlabeled sample sizes $N_s$. Let us recall the definition of $g^*_{\hat{\lambda}^{(1)}, \hat{\lambda}^{(2)}}$: conditional on the data
\begin{equation*}
g^*_{\hat{\lambda}^{(1)}, \hat{\lambda}^{(2)}} \in \arg\min_{g \in \mathcal{G}} \mathcal{R}_{\hat{\lambda}^{(1)}, \hat{\lambda}^{(2)}}(g)\enspace.
\end{equation*}
Then using similar arguments as those leading to Eq.~\eqref{eq:eqOracleLambda} implies that 
\begin{equation*}
g^*_{\hat{\lambda}^{(1)}, \hat{\lambda}^{(2)}} (x,s)\in \arg\max_{k \in [K]}  \left(\pi_s p_k(x,s)-s (\hat{\lambda}^{(1)}_k-\hat{\lambda}^{(2)}_k\right) \enspace.
\end{equation*}
As a consequence, using the writing of the fair-risk provided by Lemma~\ref{lem:eqDecompLambdaEps}
\begin{multline}
\label{eq:eqrisk3}
\mathcal{R}_{\hat{\lambda}^{(1)}, \hat{\lambda}^{(2)}}(\hat{g}) - \mathcal{R}_{\hat{\lambda}^{(1)}, \hat{\lambda}^{(2)}}(g^*_{\hat{\lambda}^{(1)}, \hat{\lambda}^{(2)}})  =\\ \sum_{s \in \mathcal{S}} \mathbb{E}_{X|S=s}\left[\max_{k \in[K]} \left(\pi_s p_k(X,s)-s(\hat{\lambda}^{(1)}_k-\hat{\lambda}^{(2)}_k)\right)-
\sum_{k=1}^K \left(\pi_s p_k(X,s)-s(\hat{\lambda}^{(1)}_k-\hat{\lambda}_k^{(2)})\right)\one_{\{\hat{g}(X,s) = k\}}\right]\enspace.
\end{multline}
Because of the indicator function, there is only one non-zero element in the inner sum. Then we observe that for each $s \in \mathcal{S}$
\begin{multline*}
\left|\max_{k \in [K]}\left(\pi_s p_k(X,s)-s(\hat{\lambda}^{(1)}_k-\hat{\lambda}_k^{(2)})\right)-
\sum_{k=1}^K \left(\pi_s p_k(X,S)-s(\hat{\lambda}^{(1)}_k-\hat{\lambda}_k^{(2)})\right)\one_{\{\hat{g}(X,s) = k\}}\right| \\ \leq  2 \max_{k \in [K]}\left| (\pi_s p_k(X,s)-s(\hat{\lambda}^{(1)}_k-\hat{\lambda}^{(2)}_k)) - (\hat{\pi}_s \bar{p}_k(X,s)-s(\hat{\lambda}^{(1)}_k-\hat{\lambda}_k^{(2)}))\right|\\
\leq 2 \left(\max_{k \in [K]}\left|p_k(X,s)-\bar{p}_{k}(X,s) \right| + \left|\pi_s-\hat{\pi}_s\right|\right)\enspace,
\end{multline*}
where the last inequality is due to the fact that  $\pi_s, \hat{\pi}_s, p_k$, and $\bar{p}_k$ are all in $[0,1]$. Therefore, recalling that $\bar{p}_k$ is a randomized version of $\hat{p}_k$ we can write
\begin{equation*}
\mathcal{R}_{\hat{\lambda}^{(1)}, \hat{\lambda}^{(2)}}(\hat{g}) - \mathcal{R}_{\hat{\lambda}^{(1)}, \hat{\lambda}^{(2)}}(g^*_{\hat{\lambda}^{(1)}, \hat{\lambda}^{(2)}}) \leq C\left(\|\mathbf{\hat{p}}-\mathbf{p}\|_1 + \sum_{s \in \mathcal{S}} |\hat{\pi}_s-\pi_s| + u\right)\enspace,
\end{equation*}
and obtain the bound 
\begin{equation*}
\mathcal{R}_{\hat{\lambda}^{(1)}, \hat{\lambda}^{(2)}}(\hat{g}) - \mathcal{R}_{\lambda^{*(1)}, \lambda^{*(2)}}\left(g^*_{\lambda^{*(1)}, \lambda^{*(2)}}\right) \leq C\left(\|\mathbf{\hat{p}}-\mathbf{p}\|_1 + \sum_{s \in \mathcal{S}} |\hat{\pi}_s-\pi_s| + u\right)\enspace.
\end{equation*}
In view of Equation~\eqref{eq:eqrisk1}, the above inequality together with Equation~\eqref{eq:eqrisk2} yield the desired result. 
\end{proof}

\begin{proof}[Proof of Theorem~\ref{thm:fastRates}]

Let us remind the reader that for each $k \in [K]$, and $s \in \mathcal{S}$
\begin{equation*}
{h}_{k}^s(X, \hat{\lambda}^{(1)}_k,\hat{\lambda}^{(2)}_k):=
\left(\pi_s p_k(X,s)-s(\hat{\lambda}^{(1)}_k-\hat{\lambda}_k^{(2)})\right).
\end{equation*}
We start the proof with Equation~\eqref{eq:eqrisk3},
\begin{multline*}
\mathcal{R}_{\hat{\lambda}^{(1)}, \hat{\lambda}^{(2)}}(\hat{g}) - \mathcal{R}_{\hat{\lambda}^{(1)}, \hat{\lambda}^{(2)}}(g^*_{\hat{\lambda}^{(1)}, \hat{\lambda}^{(2)}})  =\\ 
\sum_{s \in \mathcal{S}} \mathbb{E}_{X|S=s}\left[\max_{k \in[K]}{h}_{k}^s(X, \hat{\lambda}^{(1)}_k,\hat{\lambda}^{(2)}_k))-
\sum_{k=1}^K {h}_{k}^s(X, \hat{\lambda}^{(1)}_k,\hat{\lambda}^{(2)}_k)\one_{\{\hat{g}(X,S) = k\}}\right]\enspace.
\end{multline*}
Furthermore, we have that
\begin{equation*}
g^*_{\hat{\lambda}^{(1)}, \hat{\lambda}^{(2)}}(X,s) \in \arg\max_{k \in [K]}  {h}_{k}^s\left(X, \hat{\lambda}^{(1)}_k,\hat{\lambda}^{(2)}_k \right) \enspace.
\end{equation*}
Therefore, we observe that
\begin{multline}
\label{eq:eqrisk4}
\max_{k \in[K]} {h}_{k}^s(X, \hat{\lambda}^{(1)}_k,\hat{\lambda}^{(2)}_k))-
\sum_{k=1}^K {h}_{k}^s(X, \hat{\lambda}^{(1)}_k,\hat{\lambda}^{(2)}_k)\one_{\{\hat{g}(X,S) = k\}}  = \\\sum_{i=1, k\neq i}^K \left|{h}_{i}^s(X, \hat{\lambda}^{(1)}_i,\hat{\lambda}^{(2)}_i)-{h}_{k}^s(X, \hat{\lambda}^{(1)}_k,\hat{\lambda}^{(2)}_k) \right| \one_{\{g^*_{\hat{\lambda}^{(1)}, \hat{\lambda}^{(2)}}(X,s) =i\}} \one_{\{\hat{g}(X,s) = k\}}.
\end{multline}
Moreover, for $k \neq i$  on the event 
$\left\{g^*_{\hat{\lambda}^{(1)}, \hat{\lambda}^{(2)}}(X,s) =i, \hat{g}(X,s) = k \right\}$, we have from Equation~\eqref{eq:eqrisk3}
\begin{equation}
\label{eq:eqrisk5}
\left|{h}_{i}^s(X, \hat{\lambda}^{(1)}_i,\hat{\lambda}^{(2)}_i)-{h}_{k}^s(X, \hat{\lambda}^{(1)}_k,\hat{\lambda}^{(2)}_k) \right| 
\leq 2 \max_{s \in\mathcal{S}} \left(\max_{{k}\in [K]} \sup_{x} \left|p_k(x,s)-\bar{p}_{k}(x,s) \right| + \left|\pi_s-\hat{\pi}_s\right|\right).
\end{equation}
Now, we observe that from Assumption~\ref{ass:denistyAss}, conditional on the data for each $s \in \mathcal{S}$
\begin{multline*}
\mathbb{P}_{X|S=s}\left(\left|{h}_{i}^s(X, \hat{\lambda}^{(1)}_i,\hat{\lambda}^{(2)}_i)-{h}_{k}^s(X, \hat{\lambda}^{(1)}_k,\hat{\lambda}^{(2)}_k) \right| 
\leq 2 \left(\max_{{k}\in [K]} \sup_{x} \left|p_k(x,s)-\bar{p}_{k}(x,s) \right| + \left|\pi_s-\hat{\pi}_s\right|\right)\right) \\
\leq C \left(\max_{{k}\in [K]} \sup_{x} \left|p_k(x,s)-\bar{p}_{k}(x,s) \right| + \left|\pi_s-\hat{\pi}_s\right|\right).
\end{multline*}
Combining the above inequality with Equation~\eqref{eq:eqrisk3}, Equation~\eqref{eq:eqrisk4}, and Equation~\eqref{eq:eqrisk5}, we obtain that
\begin{multline*}
\mathcal{R}_{\hat{\lambda}^{(1)}, \hat{\lambda}^{(2)}}(\hat{g}) - \mathcal{R}_{\hat{\lambda}^{(1)}, \hat{\lambda}^{(2)}}(g^*_{\hat{\lambda}^{(1)}, \hat{\lambda}^{(2)}}) \leq C\sum_{s \in \mathcal{S}}  \left(\max_{{k}\in [K]} \sup_{x} \left|p_k(x,s)-\bar{p}_{k}(x,s) \right| + \left|\pi_s-\hat{\pi}_s\right|\right)^2 \\
\leq C  \left(\left\|\hat{\bf p} -\bf{p}\right\|^2_{\infty} + u^2 +\sum_{s \in \mathcal{S}} \left|\hat{\pi}_s -\pi_s\right|^2\right).  
\end{multline*}
Finally, we deduce again the desired result from the above inequality, Equation~\eqref{eq:eqrisk1}, and Equation~\eqref{eq:eqrisk2}.

\end{proof}

\section{Additional numerical experiments}
\label{sec:add_numres}
\begin{figure}[h!]
\begin{center}
\includegraphics[scale=0.45]{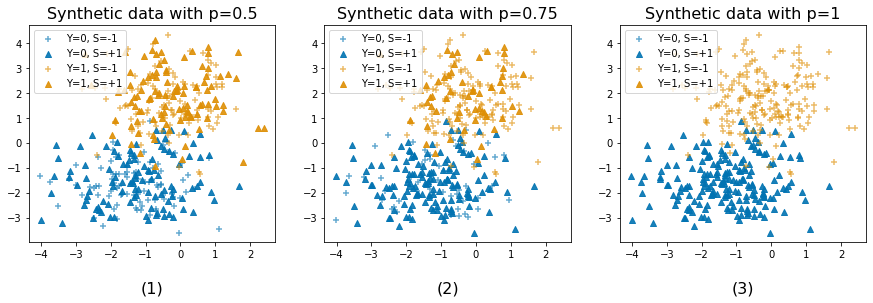}
\caption{Example of synthetic data in binary case where $d = 2$ and $m = 1$. The level of unfairness is set as follows: \textit{(1)} $p=0.5$ (no unfairness); \textit{(2)} $p=0.75$ (unfair dataset); \textit{(3)} $p=1$ (highly unfair dataset).}
\label{fig:syn_data}
\end{center}
\end{figure}

\end{document}